\documentclass[bj,preprint]{imsart}

\RequirePackage{amsthm,amsmath,amsfonts,amssymb}
\RequirePackage[numbers]{natbib}

\RequirePackage[colorlinks,citecolor=blue,urlcolor=blue]{hyperref}

\startlocaldefs
\numberwithin{equation}{section}
\theoremstyle{plain}
\newtheorem{theorem}{Theorem}[section]
\theoremstyle{plain}
\newtheorem{example}{Example}[section]
\theoremstyle{plain}

\theoremstyle{plain}
\newtheorem{proposition}{Proposition}[section]
\theoremstyle{plain}
\theoremstyle{plain}
\newtheorem{corollary}{Corollary}[section]
\newtheorem{lemma}{Lemma}[section]
\theoremstyle{plain}
\newtheorem{definition}{Definition}[section]
\theoremstyle{remark}
\newtheorem{remark}{Remark}[section]

\newcommand{\R}{\mathbb{R}}
\newcommand{\N}{\mathbb{N}}

\endlocaldefs

\begin{document}

\begin{frontmatter}
\title{Large deviations built on max-stability}
\runtitle{Large deviations built on max-stability}

\begin{aug}
\author[A]{\fnms{Michael} \snm{Kupper}\ead[label=e1,mark]{kupper@uni-konstanz.de}}
\and
\author[A]{\fnms{Jos\'{e} Miguel} \snm{Zapata}\ead[label=e2,mark]{jmzg1@um.es}}
\address[A]{Department of Mathematics and Statistics, University of Konstanz.
\printead{e1,e2}}
\end{aug}

\runauthor{M. Kupper and J.~M. Zapata}

\affiliation{University of Konstanz}

\begin{abstract}
In this paper, we show that the basic results in large deviations theory hold for general monetary risk measures, which satisfy the crucial property of max-stability. A max-stable monetary risk measure fulfills a lattice homomorphism property, and satisfies under a suitable tightness condition the Laplace Principle (LP), that is, admits a dual representation with affine convex conjugate. By replacing asymptotic concentration of probability by concentration of risk, we formulate a Large Deviation Principle (LDP) for max-stable monetary risk measures, and show its equivalence to the LP. In particular, the special case of the asymptotic entropic risk measure corresponds to the classical Varadhan-Bryc equivalence between the LDP and LP. The main results are illustrated by the asymptotic shortfall risk measure.\\~

\end{abstract}


\end{frontmatter}


\section{Introduction}
Large deviations theory covers the asymptotic concentration of probability distributions. Let $(X_n)_{n\in\mathbb{N}}$ be a sequence of random variables on a probability space $(\Omega,\mathcal{F},\mathbb{P})$ with values in a metric space $S$.
If the upper and lower \emph{asymptotic concentrations} 
\[
\underline J_A:=\underset{n\to\infty}\liminf\frac{1}{n}\log \mathbb{P}(X_n\in A)\quad\mbox{and}\quad \overline J_A:=\underset{n\to\infty}\limsup\frac{1}{n}\log \mathbb{P}(X_n\in A)
\]
satisfy $J_A=\underline{J}_A=\overline{J}_A$ for some Borel set $A$, then $\mathbb{P}(X_n\in A)\approx e^{n J_A}$ as $n\to\infty$. In other words, 
the rate of convergence of $\mathbb{P}(X_n\in A)\to 0$ is exponentially fast for $J_A<0$. Generally speaking, large deviations theory aims at finding bounds for the asymptotic  concentration $J_A$, and thus quantifying the speed of convergence. 
For instance, for the sample means 
$X_n:=\frac{1}{n}(\xi_1+\ldots+\xi_n)$ of a sequence $(\xi_k)_{k\in\mathbb{N}}$ of real-valued i.i.d.~random variables, the  Cram\'{e}r's large deviations theorem states that the \emph{Large Deviation Principle} (LDP)
\[
-\underset{x\in {\rm int}(A)}\inf I(x) \le \underline{J}_A \le \overline{J}_A \le -\underset{x\in {\rm cl}(A)}\inf I(x)
\] 
holds for all Borel sets $A$, where the  \emph{rate function} $I$ is given in that case as the convex conjugate of the
logarithmic moment generating function $\Lambda(y):=\log \mathbb{E}[e^{y \xi_1}]$. Cram\'{e}r's result \citep{cramer1938nouveau}  was generalizing some previous partial results in statistics such as  Khinchin~\citep{khintchine1929neuen} and Smirnoff~\citep{smirnoff1933uber}. 
Sanov~\citep{sanov1958probability} established an analogue of Cram\'{e}r's theorem for the empirical measures associated to a sequence of i.i.d.~random variables. After these pioneers works, large deviations theory was systematically developed by Donsker and Varadhan~\citep{donsker1,donsker2,donsker3,donsker4,varadhan1966asymptotic} and Freidlin and Wentzell~\citep{freidlin}. Of great importance is Varadhan's lemma stating that a sequence $(X_n)_{n\in\mathbb{N}}$, which satisfies the LDP, also satisfies the \emph{Laplace Principle} (LP)
\[\phi_{\rm ent}(f)=\sup_{x\in S}\{f(x)-I(x)\}\] 
for all continuous bounded functions $f$. Here, $\phi_{\rm ent}$ denotes the \emph{asymptotic entropic risk measure}, which is defined as $\phi_{\rm ent}(f):=\lim_{n\to\infty} \frac{1}{n}\log \mathbb{E}[\exp(n f(X_n))]$. A converse statement, when the LP implies the LDP, is due to Bryc, see \citep{bryc1990large}. For more historical details and a deeper look into large deviations theory, we refer to the excellent monograph~\citep{zeitouni1998large} by Dembo and Zeitouni and its extensive list of references.

The asymptotic concentration satisfies $J_A:=\inf_{r\in\R} \phi_{\rm ent}(r 1_{A^c})$,
which points out another link to the asymptotic entropic risk measure. It is a natural question to which extend large deviation theory can be developed for other risk measures rather than the entropic one. Recently, several results in this direction have been developed. For instance, F\"{o}llmer and Knispel~\citep{follmer2011entropic} studied the pooling of independent risks by means of a variant of Cram\'{e}r's theorem based on a coherent version of the entropic risk measure.   
Lacker~\citep{lacker2016non} proved a non-exponential version of Sanov's theorem for the empirical measures of i.i.d.~random variables, by replacing the underlying entropy by a (time-consistent) convex risk measure. 
The latter result was extended by Eckstein~\citep{eckstein2019extended} to the empirical measures associated to Markov chains. Backhoff, Lacker and Tangpi \citep{backhoff2018non}
obtained new asymptotic results for Brownian motion, with applications to Schilder's theorem. 

In this article, we systematically establish the connection between large deviations theory and risk analysis. Especially, we show that the basic results in 
large deviations theory hold for general monetary risk measures, which satisfy the crucial property of \emph{max-stability}. We thereby replace the notion of asymptotic concentration of probability by \emph{concentration of risk}. Recall that a  \emph{monetary risk measure} is a real-valued function $\phi$ on the space $B_b(S)$ of all bounded Borel measurable functions on a metric space $S$, which satisfies
 \begin{itemize}
	    \item[(N)] $\phi(0)=0$,
		\item[(T)] $\phi(f+c)=\phi(f)+c$ for every $f\in B_b(S)$ and $c\in\R$,
		\item[(M)] $\phi(f)\le\phi(g)$ for all $f,g\in B_b(S)$ with $f\le g$. 	
	\end{itemize}	
In mathematical finance, the space $S$ models all possible states in a financial market, $f\in B_b(S)$ describe the losses of a financial position in these states (negative losses are interpreted as gains), and the number $\phi(f)$ is understood as the
capital requirement which has to be added to the position $f$ to make it acceptable. 
For an introduction to risk measures we refer to F\"ollmer and Schied~\citep{foellmer01}. We say that a monetary risk measure  $\phi\colon B_b(S)\to \mathbb{R}$  is \emph{max-stable} if
\[\phi(f\vee g)\le \phi(f)\vee\phi(g)\quad \mbox{for all }f,g \in B_b(S).\] 
Intuitively speaking, for a max-stable monetary risk measure one cannot compensate losses by gains in different states (in contrast to expectations which average over states and thus allow for such a compensation); see Lemma \ref{lem:max-stable0} below. Max-stable monetary risk measures are automatically convex, and consequently allow for dual representations. We will show in Section \ref{sec:max-stable} that this representation on the space $C_b(S)$, containing all continuous bounded functions, is of the form \[\phi(f)=\max_{x\in S}\{f(x)-I_{\min}(x)\},\] for the minimal penalty function $I_{\min}(x):=\sup_{f\in C_b(S)}\{f(x)-\phi(f)\}$. In other words, in the dual representation of a max-stable monetary risk measure, it is enough to take the supremum over all Dirac measures, in which case the convex conjugate is affine; see Theorem \ref{thm:main1} below. A typical example of a max-stable monetary risk measure is the asymptotic entropy $\phi_{\rm ent}$ (in case that the limit does not exist for all $f\in B_b(S)$, one can consider the upper asymptotic entropic risk measure, for which the limit is replaced by the limit superior). Moreover, as shown in \citep{cattaneo}, a max-stable risk measure can be represented as a maxitive integral, which is related to the Maslov integral in idempotent analysis.

Given a monetary risk measure $\phi\colon B_b(S)\to\mathbb{R}$, we define the associated \emph{concentration function} $J$, which assigns to any Borel set $A$ the value
  \[J_A:=\underset{r\in\R}\inf \phi(r 1_{A^c}).\]
By allowing an arbitrary large gain $-r1_{A^c}$, the risk on $A^c$ is neglected and isolated on $A$.    
  Hence, the negative number $J_A$ quantifies the concentration of the risk on $A$, that is, how the monetary risk measure $\phi$ penalizes the states in $A$. We say that the monetary risk measure $\phi$ satisfies the LDP with rate function $I$ if 
\[-\inf_{x\in {\rm int}(A)}I(x) \le J_{A}\le -\inf_{x\in {\rm cl}(A)}I(x)\]
for every Borel set $A\subset S$. Recall that a function $I\colon S\to [0,+\infty]$ is called rate function if it is lower semicontinuous and not identically $+\infty$. Further, we say that the monetary risk measure $\phi$ satisfies the LP with rate function $I$, if it admits the dual representation 
\[\phi(f)=\underset{x\in S}\sup\{f(x)-I(x)\}\quad\mbox{ for all }f\in C_b(S).\]
The aim of this paper is to establish the link between LDP and LP for general monetary risk measures. However,
by using the dual representation, the LP can only hold if the monetary risk measure is max-stable (at least on $C_b(S)$). On the other hand, we show that the restriction to max-stable risk measures is enough to connect the LDP and LP, and thus establish a general large deviations theory for max-stable monetary risk measures. In Theorem \ref{thm:maxStVaradhan} we show that for max-stable monetary risk measures the LDP and LP are equivalent. In that case, the rate function is uniquely determined as the minimal penalty function.
In the spirit of Bryc's lemma, we provide in Section \ref{sec:locmax-stable} sufficient conditions for the LP to hold. We study a local version of max-stability and introduce some tightness conditions for the concentration function.  We show in Theorem \ref{th:localMaxSt} that a locally max-stable monetary risk measure $\phi$ satisfies the LP and LDP with rate function $I_{\min}$,
if the following tightness condition holds:
\[
\mbox{For every $M>0$, there exists a compact set $K\subset S$ such that $J_{K^c}\le -M.$}
\]
This condition states that the concentration $J_{K^c}$ on the complement of a suitable compact is arbitrary small; for instance, for the asymptotic entropic risk measure the condition implies that the sequence $(X_n)_{n\in\mathbb{N}}$ is exponentially tight. In order to derive bounds for upper and lower concentration functions, we consider in Section \ref{sec:pairs} two monetary risk measures $\underline{\phi},\,\overline{\phi}$ on $B_b(S)$ such that $\overline{\phi}$ is max-stable and $\underline{\phi}\le\overline{\phi}$. Then, we show in Proposition \ref{prop:extendedLDP} that the respective concentration functions $\underline{J}\le\overline{J}$ satisfy the same LDP with rate funtion $I$, that is
\[
-\inf_{x\in \textnormal{int} (A)}{I}(x)\le \underline{J}_{A}\le \overline{J}_A\le -\inf_{x\in \textnormal{cl} (A)}{I}(x)
\]
for all Borel sets $A\subset S$, if and only if $\underline{\phi}=\overline{\phi}$ and satisfy the LP with rate function $I$. 

The theoretical framework is illustrated with the class of \emph{asymptotic shortfall risk measures}. For the shortfall risk measure $Z\mapsto \inf\{m\in\mathbb{R}\colon\mathbb{E}[l(Z-m)]\le 1\}$, we define the associated upper asymptotic shortfall by
\[
\overline{\phi}(f)=\limsup_{n\to\infty} \frac{1}{n} \inf\big\{m\in\mathbb{R}\colon\mathbb{E}[l(nf(X_n)-m)]\le 1\big\}\quad \mbox{for all }f\in B_b(S).
\] 
Here, $l$ is a non-decreasing loss function, which penalizes the losses of the positions $f(X_n)$, which are scaled by the risk aversion coefficient $n$; see \citep{foellmer01} for a discussion on shortfall risk measures. For instance, for the exponential loss function,
the upper asymptotic shortfall coincides with the upper asymptotic entropic risk measure. Also, the lower asymptotic shortfall is defined analogously (the limit superior is replaced by the limit inferior).
Under suitable assumptions on the loss function, we show that the asymptotic shortfalls are locally max-stable, and provide formulas for the concentration functions and rate functions, respectively. Hence, under an adjusted tightness condition, the upper/lower shortfall risk measures satisfy the LDP and LP.  

The structure of the paper is as follows.  In Section \ref{sec:max-stable}, we  provide a dual representation result for max-stable monetary risk measures on $C_b(S)$, characterize the respective convex conjugates and provide their maxitive integral representation. In Section \ref{sec:LDP}, we introduce the max-stable version of the Varadhan's Large Deviation Principle  and establish the equivalence between the LP and LDP for max-stable monetary risk measures. 
In Section \ref{sec:locmax-stable}, we study local max-stable risk measures and prove a version of Bryc's lemma for locally max-stable monetary risk measures. Section \ref{sec:pairs} contains some comparison results for pairs of  max-stable monetary risk measures. Finally, Section \ref{sec:examples} is devoted to asymptotic shortfall risk measures.

\section{Max-stable monetary risk measures}\label{sec:max-stable}
Let $(S,d)$ be a non-empty metric space
with Borel $\sigma$-algebra $\mathcal{B}(S)$.
In the following, we denote by $\mathcal{X}$ a linear subspace of $\mathbb{R}^S$ which contains the constants and  $f\vee g\in \mathcal{X}$ for all $f,g\in \mathcal{X}$.\footnote{Here, constants are identified with the constant function $S\to\mathbb{R}$, $x\mapsto c$ for all $c\in\mathbb{R}$. Further, $(f\vee g)(x):=\max\{f(x),g(x)\}$ for all $x\in S$.} Typical examples for $\mathcal{X}$ include the space
$B_b(S)$ of all  Borel measurable bounded functions, or the space $C_b(S)$ of all continuous bounded functions. 
\begin{definition}
	A function $\phi\colon \mathcal{X}\to \mathbb{R}$ is called a \emph{monetary risk measure} if  
	\begin{itemize}
		\item[(N)] $\phi(0)=0$,
		\item[(T)] $\phi(f+c)=\phi(f)+c$ for all $f\in \mathcal{X}$ and $c\in\mathbb{R}$,
		\item[(M)] $\phi(f)\le\phi(g)$ for all $f,g\in \mathcal{X}$ with $f\le g$. 	
	\end{itemize}	
A monetary risk measure which in addition is convex\footnote{That is, $\phi(\lambda f+(1-\lambda)g)\le\lambda\phi(f)+(1-\lambda)\phi(g)$ for all $f,g\in \mathcal{X}$ and $\lambda\in[0,1]$.} is called a \emph{convex risk measure}.
\end{definition}
A monetary risk measure $\phi\colon \mathcal{X}\to \mathbb{R}$ is uniquely determined through its acceptance set, which is defined as   
$\mathcal{A}:=\{f\in \mathcal{X}\colon \phi(f)\le 0\}$.\footnote{Indeed, by the translation property (T),  $\phi(f)=\inf\{c\in\mathbb{R}\colon\phi(f)\le c\}=\inf\{c\in\mathbb{R}\colon\phi(f-c)\le 0\}=\inf\{c\in\mathbb{R}\colon f-c\in\mathcal{A}\}$.}  

\begin{definition}
	We call a monetary risk measure $\phi\colon \mathcal{X}\to \mathbb{R}$ \emph{max-stable}  if  \[\phi(f\vee g)\leq \phi(f)\vee\phi(g)\quad \mbox{for all } f,g\in \mathcal{X}.\]
\end{definition}

Alternatively, the notion of max-stability can be formulated as in the following lemma.
The second condition means that $\phi$ is a lattice homomorphism. The fourth condition
has the interpretation that for a max-stable monetary risk measure one cannot compensate losses by large profits in other states.

\begin{lemma}\label{lem:max-stable0} 
	For a monetary risk measure $\phi\colon\mathcal{X}\to\mathbb{R}$, the following are equivalent:
	\begin{itemize}
		\item[(i)] $\phi$ is max-stable.
		\item[(ii)] $\phi(f\vee g)= \phi(f)\vee \phi(g)$ for all $f,g\in\mathcal{X}$.
		\item[(iii)] $f\vee g\in\mathcal{A}$ for all $f,g\in\mathcal{A}$.
	\end{itemize}
If $\mathcal{X}=B_b(S)$ then the  conditions \textnormal{(i)}--\textnormal{(iii)}  are equivalent to
\begin{itemize}
\item[(iv)] $f 1_B + c 1_{B^c}\in\mathcal{A}$ and $f 1_{B^c} + c 1_{B}\in\mathcal{A}$ for $c\in\mathbb{R}$ and $B\in \mathcal{B}(S)$ implies $f\in\mathcal{A}$.  
\end{itemize} 	
\end{lemma}
\begin{proof}
	$\textnormal{(i)}\Rightarrow \textnormal{(ii)}$: By monotonicity $\rm{(M)}$, it holds 
	$\phi(f\vee g)\le\phi(f)\vee\phi(g)\le\phi(f\vee g)$.

	$\textrm{(ii)}\Rightarrow \textrm{(iii)}$: For $f,g\in\mathcal{A}$, it holds $\phi(f\vee g)=\phi(f)\vee\phi(g)\le 0$, so that $f\vee g\in\mathcal{A}$.
	
	$\textrm{(iii)}\Rightarrow\textrm{(i)}$: Let $f,g\in \mathcal{X}$ and set $c:=\phi(f)\vee\phi(g)$.
	By the translation property (T),  $f-c,g-c\in\mathcal{A}$, which shows that
	\[
	\phi\big((f\vee g)-c\big)=\phi\big((f-c)\vee(g-c)\big)\le 0.
	\]
	Again, by the translation property (T), we get $\phi(f\vee g)\le c=\phi(f)\vee\phi(g)$. 

Finally, we assume that $\mathcal{X}=B_b(S)$ and prove $\textrm{(iii)}\Leftrightarrow\textrm{(iv)}$.

$\textrm{(iii)}\Rightarrow\textrm{(iv)}$: Suppose that $f 1_B + c 1_{B^c}\in\mathcal{A}$ and $f 1_{B^c} + c 1_{B}\in\mathcal{A}$ for $c\in\mathbb{R}$ and $B\in\mathcal{B}(S)$. 
Since $f\vee c=(f 1_B + c 1_{B^c})\vee(f 1_{B^c} + c 1_{B})$ it follows from $\textrm{(iii)}$ that 
$f\vee c\in\mathcal{A}$. Hence, by monotonicity (M) 
we get $\phi(f)\le\phi(f\vee c)\le 0$, and therefore $f\in\mathcal{A}$. 

$\textrm{(iv)}\Rightarrow\textrm{(iii)}$: Let $f,g\in\mathcal{A}$ and fix $c\in \mathbb{R}$ with $c\le f$ and $c\le g$. Then, by monotonicity (M) it follows that  
$\phi((f\vee g) 1_B + c 1_{B^c})\le \phi(f)\le 0$, where $B:=\{f\ge g\}$. This shows that $(f\vee g) 1_B + c 1_{B^c}\in \mathcal{A}$, and similarly  
$(f\vee g) 1_{B^c} + c 1_{B}\in \mathcal{A}$, which by 
$\textrm{(iv)}$ implies that $f\vee g\in\mathcal{A}$. 
\end{proof}	

\begin{proposition}\label{prop:convex}
	Every max-stable monetary risk measure $\phi\colon \mathcal{X}\to \mathbb{R}$ is convex.
\end{proposition}
\begin{proof}
	For $f,g\in \mathcal{X}$ and $\lambda\in [0,1]$, it holds
	\begin{align*}
	\phi\big(\lambda f + (1-\lambda)g\big) - \lambda\phi(f) - (1-\lambda)\phi(g)
	&=\phi\big(\lambda (f-\phi(f)) + (1-\lambda)(g-\phi(g))\big)\\
	&\le\phi\big((f-\phi(f))\vee(g-\phi(g))\big)\\
	&\le\phi(f-\phi(f))\vee\phi(g-\phi(g))=0,
	\end{align*}
	where the second inequality follows from max-stability. 
\end{proof}
\subsection{Dual representation of max-stable monetary risk measures}
Notice that Proposition~\ref{prop:convex} allows for convex duality and thus for a dual characterization of max-stability. For a max-stable monetary risk measure $\phi\colon C_b(S)\to\mathbb{R}$, we define its \emph{convex conjugate} $\phi^\ast\colon\mathop{\rm ca^+_1}(S)\to[0,+\infty]$ by
\[
\phi^\ast(\mu):=\sup_{f\in C_b(S)}\Big\{\int f\,{\rm d}\mu-\phi(f)\Big\},
\]
where ${\rm ca^+_1}(S)$ denotes the set of all probability measures on $\mathcal{B}(S)$. Especially, for the Dirac measure $\delta_x$, for $x\in S$, it holds 
\begin{equation}\label{eq:ratemin}
I_{\rm min}(x):=\phi^\ast(\delta_x)=\sup_{f\in C_b(S)}\{f(x)-\phi(f)\}=\sup_{f\in\mathcal{A}} f(x),
\end{equation}
where in the last equality it is used that
$f-\phi(f)\in\mathcal{A}$ for all $f\in C_b(S)$ by the translation property (T). In particular, the function $S\to[0,+\infty]$, $x\mapsto I_{\rm min}(x)$ is lower-semicontinuous. On the other hand, even though the conjugate $\mu\mapsto \phi^\ast(\mu)$ is convex,  it does not follow that $x\mapsto I_{\rm min}(x)$ is  convex (in case that $S$ is a convex set). 

Now we are ready to state our first main result. 

\begin{theorem}\label{thm:main1} 
Suppose that $(S,d)$ is separable and let $\phi\colon C_b(S)\to\mathbb{R}$ be a monetary risk measures which is continuous from above\footnote{That is, for every sequence $(f_n)_{n\in\mathbb{N}}$ in $C_b(S)$ such that $f_n\downarrow 0$, it holds $\phi(f_n)\downarrow 0$.}. 
	Then, the following are equivalent:
	\begin{itemize}
		\item[(i)] $\phi$ is max-stable.
		\item[(ii)] $f\vee g\in \mathcal{A}$ for all $f,g\in \mathcal{A}$.
		\item[(iii)] $\phi(f)=\max_{x\in S} \{ f(x) - I_{\rm min}(x)\}$ for all $f\in C_b(S)$.
		\item[(iv)] $\phi$ is convex and $\phi^\ast(\mu)=\int \phi^\ast(\delta_x)\,\mu({\rm d}x)$ for all $\mu\in\mathop{\rm ca}^+_1(S)$.
		\item[(v)] There is a function $\gamma\colon S\to (-\infty,+\infty]$ such that $\mathcal{A}=\{f\in C_b(S)\colon f\leq \gamma\}$.
	\end{itemize}
In this case, the function $S\to [0,+\infty]$, $x\mapsto I_{\rm min}(x)$ has compact sublevel sets.   
\end{theorem}

\begin{proof} Before proving these equivalences, we show that 
\begin{equation}\label{eq:lem}	\phi^\ast(\delta_\cdot)={\rm ess.sup}_\mu \mathcal{A}\quad\mu\mbox{-almost surely}\end{equation}
	for all $\mu\in{\rm ca}^+_1(S)$, where ${\rm ess.sup}_\mu \mathcal{A}$  denotes the essential supremum of $\mathcal{A}$ w.r.t.~the probability measure $\mu$, see~\citep[Section A.5]{foellmer01}.
	To that end, fix $\mu\in {\rm ca}^+_1(S)$, and let $\xi$ be a measurable $\mu$-a.s.~representative of ${\rm ess.sup}_\mu\mathcal{A}$. 
	Since $\phi^\ast(\delta_\cdot)$ is lower semicontinuous, the function $\phi^\ast(\delta_\cdot)=\sup\mathcal{A}$ is a Borel measurable upper bound of $\mathcal{A}$.  
	Therefore, it holds $\xi\le\phi^\ast(\delta_\cdot)$ $\mu$-a.s..  
	By contradiction, suppose that $\mu(\xi<\phi^\ast(\delta_\cdot))>0$. 
	Then, there exists a rational $q\in\mathbb{Q}$ such that 
	$\mu(C)>0$, where $C:=\{\xi< q< \phi^\ast(\delta_\cdot)\}$. 
	Since $(S,d)$ is second-countable, 
	there exists $x\in S$ such that $\mu(C\cap U)>0$ for every open neighborhood $U$ of $x$.
	Fix $f\in\mathcal{A}$ with $q<f(x)\le \phi^\ast(\delta_x)$. 
	Note that such an $f$ exists because $\phi^\ast(\delta_x)=\sup_{f\in\mathcal{A}} f(x)$; see \eqref{eq:ratemin}. 
	Then, $\{q<f\}$ is an open neighborhood of $x$, and therefore $\mu(C\cap \{q<f\})>0$. 
	This shows that $\xi(y)<f(y)$ for all $y\in C\cap  \{q<f\}$, but this contradicts that $f\le\xi$ $\mu$-a.s. This shows \eqref{eq:lem}. 
	
$\textrm{(i)}\Leftrightarrow\textrm{(ii)}$: This follows from Lemma \ref{lem:max-stable0}.
	
	$\textrm{(i)} \Rightarrow \textrm{(iv)}$: 
	Since $\phi$ is max-stable it is convex by Proposition \ref{prop:convex}.
	Fix $\mu\in\mathop{\rm ca}^+_1(S)$. 
	On the one hand, since $\mathcal{A}$ is directed upwards there exists a sequence $(f_n)_{n\in\mathbb{N}}$ in $\mathcal{A}$ such that $0\le f_n\le f_{n+1}$ for all $n\in\mathbb{N}$ and
	\[f_n\to {\rm ess.sup}_\mu \mathcal{A}\quad \mu\mbox{-a.s..}\]  
	By \eqref{eq:lem} we have that ${\rm ess.sup}_\mu \mathcal{A}=\phi^\ast(\delta_\cdot)$ $\mu$-a.s.. 	
	Then, by monotone convergence, we obtain 
	\begin{align*}
	\phi^\ast(\mu)&=\sup_{f\in\mathcal{A}}\int f(x)\, \mu({\rm d}x)\geq \sup_{n\in\mathbb{N}}\int f_n(x)\, \mu({\rm d}x)=\int \phi^\ast(\delta_x)\, \mu({\rm d}x).
	\end{align*}
	On the other hand,  
	\[
	\phi^\ast(\mu)=\sup_{f\in\mathcal{A}}\int f(x) \,\mu({\rm d}x) \leq \int {\rm ess.sup}_\mu \mathcal{A} \, \mu({\rm d}x)=
	\int  \phi^\ast(\delta_x)\, \mu({\rm d}x).
	\]
	
	$\textrm{(iv)}\Rightarrow \textrm{(iii)}$: Fix $f\in C_b(S)$.
	By the non-linear version of the Daniell-Stone representation theorem (see \citep[Proposition 1.1]{cheridito2015representation}) we have
	\begin{align*}
	\phi(f)&=\max_{\mu\in\mathop{\rm ca}^+_1(S)}\Big\{ \int f(x)\,\mu({\rm d}x) - \phi^\ast(\mu) \Big\}\\
	&=\max_{\mu\in\mathop{\rm ca}^+_1(S)}\Big\{ \int f(x)\,\mu({\rm d}x) - \int \phi^\ast(\delta_x)\,\mu({\rm d}x) \Big\}\\
	&=\max_{\mu\in\mathop{\rm ca}^+_1(S)}\Big\{ \int \big(f(x)- \phi^\ast(\delta_x)\big)\,\mu({\rm d}x) \Big\}\\
	&=\max_{x\in S}\big\{ f(x) - \phi^\ast(\delta_x ) \big\}.
	\end{align*}
	Moreover, by the arguments in the proof of \citep[Theorem 2.2]{bartl2019robust} it follows that the sublevel sets $\{\mu\in \textnormal{ca}^1_+(S)\colon \phi^\ast(\mu)\le r\}$ are $\sigma(\textnormal{ca}^1_+(S),C_b(S))$-compact
	for all $r\in\mathbb{R}$. As a consequence, the function $S\to [0,+\infty]$, $x\mapsto \phi^\ast(\delta_x)$ has compact sublevel sets.   
	
	$\textrm{(iii)}\Rightarrow \textrm{(v)}$: From the dual representation we obtain \[\mathcal{A}=\{f\in C_b(S)\colon f(x)\leq\phi^\ast(\delta_x)\mbox{ for all } x\in S\}.\]  
	
	$\textrm{(v)}\Rightarrow\textrm{(ii)}$: 
	For $f,g\in\mathcal{A}$ it holds $f, g\leq \gamma$, so that 
	\[
	(f\vee g)(x)=f(x)\vee g(x)\leq \gamma(x)\quad\mbox{for all }x\in S.
	\]
	This shows that $f\vee g\in\mathcal{A}$.
\end{proof}

As a consequence, we get the following result. 
\begin{corollary}\label{cor:compact}
	Let $K\subset S$ be compact. Every max-stable monetary risk measure $\phi\colon C(K)\to\mathbb{R}$ has the representation
	$\phi(f)=\max_{x\in S} \{ f(x) - I_{\min}(x)\}$ for all $f\in C(K)$.\footnote{We denote by $C(K)$ the space of all continuous functions $f\colon K\to\mathbb{R}$.}
\end{corollary}
\begin{proof}
	Since every compact $K\subset S$ is separable, and $\phi$ is continuous from above by Dini's lemma, the statement follows from Theorem \ref{thm:main1}.
\end{proof}	
	
\begin{remark}\label{rem:Imin}
Based on standard arguments in convex duality, the function $I_{\rm min}$ is minimal among those functions $I\colon S\to[0,+\infty]$, for which $\phi$ admits the representation
\begin{equation}\label{eq:repArbRate}
	\phi(f)=\underset{x\in S}\sup\{f(x)-I(x)\}\quad\mbox{ for all }f\in C_b(S).
	\end{equation}
Also, if $\phi$ has the representation	\eqref{eq:repArbRate}, then $I$ can be replaced by $I_{\min}$. Indeed, for every $f\in C_b(S)$ and $x\in S$, it follows from~\eqref{eq:repArbRate} that 
	$I(x)\ge f(x)-\phi(f)$,   
	which shows that $I\ge I_{\rm min}$. 
	Hence, for $f\in C_b(S)$, we obtain 
	\[
	\phi(f)=\underset{x\in S}\sup\{f(x)-I(x)\}\le\underset{x\in S}\sup\{f(x)-I_{\rm min}(x)\}\le \phi(f), 
	\] 
	where the last inequality follows by definition of $I_{\rm min}$. 
\end{remark}
\subsection{Maxitive integral representation of max-stable monetary risk measures} 
  
Given a monetary risk measure $\phi\colon B_b(S)\to\R$, we define 
the \emph{concentration function} $J\colon\mathcal{B}(S)\to [-\infty,0]$ of the monetary risk measure $\phi$ by
\begin{equation}\label{eq:JA}
J_A:=\inf_{r\in \mathbb{R}}\phi(r 1_{A^c})\quad\mbox{for all }A\in \mathcal{B}(S).
\end{equation}
Directly from the definition, it holds $J_A\le J_{B}$ whenever $A\subset B$.  The concentration function is connected to the acceptance set $\mathcal{A}=\{f\in B_b(S)\colon\phi(f)\le 0\}$ as follows. 
\begin{lemma}
For every $A\in\mathcal{B}(S)$ and $s\in \R$, it holds
\begin{itemize}
\item[(i)] $s>-J_A$ implies that $s 1_A - r 1_{A^c}\notin\mathcal{A}$ for all $r>0$,
\item[(ii)] $s<-J_A$ implies that there exists $r>0$ such that $s 1_A - r 1_{A^c}\in\mathcal{A}$.
\end{itemize}
In particular, $-J_A=\sup\{s\ge 0 \colon \mbox{there exists } r>0 \mbox{ such that }s 1_A - r 1_{A^c}\in\mathcal{A}\}$.
\end{lemma}
\begin{proof}
For every $r> 0$, it holds
\begin{equation}
\label{eq:translationConc}
\phi(s 1_A - r 1_{A^c})=\phi(-(s+r)1_{A^c})+s\quad\mbox{for all }s\in\R.
\end{equation}

If $s>-J_A$, then it follows from \eqref{eq:translationConc} that 
$\phi(s 1_A - r 1_{A^c})\ge J_A + s>0$, and therefore $s 1_A - r 1_{A^c}\notin\mathcal{A}$.   
On the other hand, due to \eqref{eq:translationConc}, we have that $\phi(s 1_A - r 1_{A^c})\downarrow J_A+s$ as $r\to+\infty$. 
If $s<-J_A$, then $J_A+s<0$, and therefore $\phi(s 1_A - r 1_{A^c})\le 0$ for $r$ large enough. 
\end{proof}

\begin{lemma}\label{lem:max-stableJ}
Suppose that $\phi$ is max-stable.
 Then $J_{A_1\cup\ldots\cup A_N}\le J_{A_1}\vee\ldots\vee J_{A_N}$ for all  $A_1,\ldots,A_N\in\mathcal{B}(S)$. 
\end{lemma}
\begin{proof}

For every $r\le 0$, it holds from max-stability that 
\[
\phi(r 1_{(A_1\cup\ldots\cup A_N)^c})=\phi(r 1_{A^c_1}\vee\ldots\vee r 1_{A^c_N})\le \phi(r 1_{A^c_1})\vee\ldots\vee \phi(r 1_{A^c_N}).
\]
Hence, by letting $r\to -\infty$ we obtain $J_{A_1\cup\ldots\cup A_N}\le J_{A_1}\vee\ldots\vee J_{A_N}$. 
\end{proof}

A function $\mu\colon\mathcal{B}(S)\to[-\infty,0]$ is called a \emph{penalty} if $\mu(\emptyset)=-\infty$, $\mu(S)=0$,  and $\mu(A)\le\mu(B)$ whenever $A\subset B$.  
A penalty $\mu$ is said to be max-stable if $\mu(A\cup B)\le \mu(A)\vee\mu(B)$. 
Given a max-stable penalty $\mu\colon\mathcal{B}(S)\to[-\infty,0]$, the \emph{maxitive integral} of $f\in B_b(S)$ with respect to $\mu$ is defined by
\begin{equation}\label{eq:maxitiveInt}
\int^X f\, {\rm d}\mu:=\underset{r\in \R}\sup\{r + \mu(\{f>r\})\}.
\end{equation} 
The maxitive integral~\eqref{eq:maxitiveInt} was introduced in~\citep{cattaneo}. 
It is closely related to the Maslov integral in idempotent analysis~\citep{maslov}, and can be obtain as a transformation of the Shilkret integral~\citep{shilkret}.  
Moreover, in the context of idempotent analysis, max-stable risk measures can be identified with `linear' functionals. 

The following result is an adaptation to the present setting of~\citep[Corollary 6]{cattaneo}. In contrast to \citep[Corollary 6]{cattaneo}, we assume max-stability on $B_b(S)$ rather than the space of all functions $f\colon S\to [-\infty,+\infty)$ which are bounded from above.
\begin{proposition}\label{prop:charactMaxSt}
Let $\phi\colon B_b(S)\to\R$ be a function.  
Then, the following conditions are equivalent:
\begin{itemize}
\item[(i)] $\phi$ is a max-stable monetary risk measure.
\item[(ii)] There exists a max-stable penalty $\mu\colon\mathcal{B}(S)\to[-\infty,0]$ such that
\[
\phi(f)=\int^X f \,{\rm d}\mu\quad\mbox{ for all }f\in B_b(S).
\]
\item[(iii)] The concentration function $J\colon\mathcal{B}(S)\to[-\infty,0]$ is a max-stable penalty  and  
\begin{equation}\label{eq:intRep}
\phi(f)=\int^X f \,{\rm d} J\quad\mbox{ for all }f\in B_b(S).
\end{equation}
\end{itemize}
\end{proposition}
\begin{proof}
$(iii)\Rightarrow(ii)$: The implication is obvious.

$(ii)\Rightarrow(i)$: This follows directly by definition of the  maxitive integral and since the penalty $\mu$ is max-stable.

$(i)\Rightarrow(iii)$: The concentration $J$ is a penalty, which is max-stable by  Lemma~\ref{lem:max-stableJ}. To show \eqref{eq:intRep}, we first assume that $f$ is simple, that is, $f=\sum_{i=1}^N \alpha_i f^{-1}(\alpha_i)$ with $\alpha_1<\alpha_2<\ldots<\alpha_N$.  
Then, we have
\[
f=\underset{1\le i\le N}\vee \left\{\alpha_i 1_{\{f> \alpha_{i-1}\}} + r 1_{\{f\le \alpha_{i-1}\}} \right\}   
\]
for $r\in\mathbb{R}$ small enough, where we define $\alpha_{0}:=-\infty$. 
Since $\phi$ is max-stable, it follows that
$$\phi(f)=\underset{1\le i\le N}\vee \phi\left(\alpha_i 1_{\{f> \alpha_{i-1}\}} + r 1_{\{f\le \alpha_{i-1}\}} \right)=\underset{1\le i\le N}\vee\left\{\alpha_{i} + \phi((r-\alpha_i) 1_{\{f\le\alpha_{i-1}\}}) \right\}.   
$$
Letting $r\to-\infty$, we get
\begin{align*}
\phi(f)&=\underset{1\le i\le N}\vee\big\{\alpha_i + J_{\{f> \alpha_{i-1}\}} \big\}=\underset{1\le i\le N}\vee \underset{r\in [\alpha_{i-1},\alpha_i)}\sup\big\{r + J_{\{f>r\}} \big\}\\&=\underset{r\in \R}\sup\big\{r + J_{\{f>r\}} \big\}
=\int^X f \,{\rm d} J.
\end{align*}
Finally, for general $f\in B_b(S)$, the statement follows by approximating
$f$ with simple functions, and since $\phi$ is Lipschitz continuous by the monotonicity (M) and the translation property (T).  
\end{proof}
\begin{remark}
Suppose that $\phi_1,\phi_2\colon B_b(S)\to\R$ are max-stable monetary risk measures with respective concentration functions $J_1,J_2$. 
If $J_1(A)=J_2(A)$ for all $A\in\mathcal{B}(S)$, then it follows from~\eqref{eq:intRep} in  Proposition~\ref{prop:charactMaxSt} that $\phi_1(f)=\phi_2(f)$ for all $f\in B_b(S)$. 
\end{remark} 
Notice that the results in this subsection hold for general topological spaces $S$.

\subsection{Examples}\label{ex:maxSt}

	We provide some examples of 
	max-stable monetary risk measures.
\begin{itemize}
\item[(a)] Let  $\gamma\colon S\to [0,+\infty]$ be a function that is not identically $+\infty$. Then, 
\[
\phi\colon B_b(S)\to \R,\quad \phi(f):=\underset{x\in S}\sup\{f(x)-\gamma(x)\}
\]
is a max-stable monetary risk measure.
\item[(b)] Let $(\phi_i)$ be a family of max-stable monetary risk measures on $B_b(S)$. 
Then, $\vee_i\phi_i$ is a max-stable monetary risk measure on $B_b(S)$.
\item[(c)] 
Let $(X_n)_{n\in\mathbb{N}}$ be a sequence of $S$-valued random variables defined on a probability space $(\Omega,\mathcal{F},\mathbb{P})$. 
The \emph{upper asymptotic entropy} 
\[
\overline{\phi}_{\rm ent}(f):=\limsup_{n\to\infty}\frac{1}{n}\log\mathbb{E}[\exp(n f(X_n))]
\]
is a max-stable monetary risk measure on $B_b(S)$.  
Indeed, it follows directly from the definition
that $\overline{\phi}_{\rm ent}$ is a monetary risk measure. As for the max-stability, for $f,g\in B_b(S)$ and $A_n:=\{f(X_n)\ge g(X_n)\}$, it holds
\begin{align*}
&\overline{\phi}_{\rm ent}(f\vee g)=\limsup_{n\to\infty} \frac{1}{n} \log \Big(\mathbb{E}\big[\exp(n f(X_n)\vee n g(X_n)\big]\Big)\\ =&\limsup_{n\to\infty} \frac{1}{n} \log \Big(\mathbb{E}\big[\exp(n f(X_n))1_{A_n}\big]+ \mathbb{E}\big[\exp(n g(X_n))1_{A_n^c}\big]\Big)\\
\le& \limsup_{n\to\infty} \frac{1}{n} \log \mathbb{E}\big[\exp(n f(X_n))1_{A_n}\big]\vee \limsup_{n\to\infty} \frac{1}{n} \log  \mathbb{E}\big[\exp(n g(X_n))1_{A_n^c}\big]\\
\le& \overline{\phi}_{\rm ent}(f)\vee \overline{\phi}_{\rm ent}(g).	
\end{align*}	
Along the same line of argumentation, one can show that the \emph{lower asymptotic entropy} 
$\underline{\phi}_{\rm ent}(f):=\liminf_{n\to\infty}\frac{1}{n}\log\mathbb{E}[\exp(n f(X_n))]$ is a max-stable monetary risk measure on $B_b(S)$.

\item[(d)] In the context of the previous example, for a non-empty family $\mathcal{P}$ of probability measures on $\mathcal{F}$, we define the 
\emph{upper} and \emph{lower robust asymptotic entropies} by
\[
\limsup_{n\to\infty}\frac{1}{n}\log\big(\sup_{\mathbb{P}\in\mathcal{P}}\mathbb{E}_{\mathbb{P}}[\exp(n f(X_n))]\big)\quad\mbox{and}\quad
\liminf_{n\to\infty}\frac{1}{n}\log\big(\sup_{\mathbb{P}\in\mathcal{P}}\mathbb{E}_{\mathbb{P}}[\exp(n f(X_n))]\big)
\]
for all $f\in B_b(S)$. By the same arguments as in the previous example, it follows that the upper/lower robust asymptotic entropy is a max-stable monetary risk measure on $B_b(S)$. Thereby, it should be noted that the nonlinear expectation $\sup_{\mathbb{P}\in\mathcal{P}}\mathbb{E}_{\mathbb{P}}[\cdot]$ is subadditive. For instance, $\mathcal{P}$ could be modeled as the Wasserstein ball $\mathcal{P}=\{\mathbb{P}\colon \mathcal{W}(\mathbb{P},\mathbb{P}_0)\le c\}$ consisting of all probability measures $\mathbb{P}$ which in some Wasserstein distance $\mathcal{W}$ are close to some reference probability measure $\mathbb{P}_0$.
\end{itemize}

\section{Large deviations built on max-stable monetary risk measures}\label{sec:LDP}

Throughout this section, let $\phi\colon B_b(S)\to\mathbb{R}$ be a monetary risk measure. 
We will provide sufficient conditions such that the monetary risk measure $\phi$ satisfies the \emph{Laplace principle (LP)} with \emph{rate function}\footnote{That is, a lower semicontinuous function $I\colon S\to[0,+\infty]$, which is not identically $+\infty$.} $I$, i.e., $\phi$ has the dual representation
\begin{equation}\label{eq:representation}
\phi(f)=\underset{x\in S}\sup\{f(x)-I(x)\}\quad\mbox{ for all }f\in C_b(S).
\end{equation}
We define the  \emph{minimal rate function} $I_{\min}\colon S\to[0,+\infty]$ of the monetary risk measure $\phi$ by 
\begin{equation}\label{eq:Imin}
I_{\rm min}(x):=(\phi|_{C_b(S)})^\ast(\delta_x)=\underset{f\in C_b(S)}\sup\{f(x)-\phi(f)\},
\end{equation}
and the \emph{concentration function} $J\colon\mathcal{B}(S)\to [-\infty,0]$ is defined as in \eqref{eq:JA}.
 
The following result is an analogue of the probabilistic representation of the rate function in large deviation theory, see e.g.~\citep[Remark 2.1]{dinwoodie} and \citep[Theorem 4.1.18]{zeitouni1998large}. 
\begin{proposition}\label{prop:ratefunction}
For every $x\in S$, it holds
	\[
	-I_{\rm min}(x)=\lim_{\delta\downarrow 0}J_{B_\delta(x)},
	\]
	where $B_\delta(x):=\{y\in S\colon d(x,y)<\delta\}$. 	
\end{proposition}
\begin{proof}
	Let $x\in S$. Fix $r\le 0$ and $\delta>0$. 
	By Urysohn's lemma
	there exists a continuous function $f_{r,\delta}:S\to [r,0]$ with $f_{r,\delta}(x)=0$ and $B_\delta(x)^c\subset f_{r,\delta}^{-1}(r)$. Then, we have
	\[
	I_{\rm min}(x)\ge -\phi(f_{r,\delta})\ge -\phi(r 1_{B_\delta(x)^c}),
	\]
	and therefore
	\[
	I_{\rm min}(x)\ge \sup_{r\ge 0}-\phi(r 1_{B_\delta(x)^c})= -J_{B_\delta(x)},
	\]
	which shows that $-I_{\rm min}(x)\le\lim_{\delta\downarrow 0}J_{B_\delta(x)}$. 
	
	As for the other inequality, let $f\in C_b(S)$.
	For every $\varepsilon>0$ there exists $\delta>0$ so that $f(z)\geq f(x)-\varepsilon$ for all $z\in B_\delta(x)$. Then, for $r$ large enough, it holds
	\begin{align*}
	\phi(f)&\ge \phi( (f(x)-\varepsilon)1_{B_\delta(x)} - r 1_{B_\delta(x)^c} )\\&=
	\phi( f(x)-\varepsilon - (r + f(x)+\varepsilon)  1_{B_\delta(x)^c} )\\
	&=f(x)-\varepsilon + \phi(-(r + f(x)+\varepsilon)  1_{B_\delta(x)^c}).
	\end{align*}
	By taking the infimum over $r>0$, we obtain
	\[
	\phi(f)\ge f(x)-\varepsilon + J_{B_\delta(x)}.
	\]
	Since $\varepsilon>0$ was arbitrary, we get
	\[
	\phi(f)\ge f(x) + \lim_{\delta \downarrow 0}J_{B_\delta(x)}.
	\]
	This shows that 
	\[
	-\lim_{\delta \downarrow 0}J_{B_\delta(x)}\ge  f(x) - \phi(f)
	\]
	for all $f\in C_b(S)$, and therefore
	$\lim_{\delta \downarrow 0}J_{B_\delta(x)}\le -I_{\rm min}(x)$. The proof is complete.
\end{proof}

\begin{corollary}\label{cor:LDPlower}
It holds
\begin{equation}\label{LDP1}
-\inf_{x\in A}I_{\rm min}(x)\le J_A\quad\mbox{for all  }A\subset S\mbox{ open}.
\end{equation}
If in addition, $\phi$ is max-stable, then 
\begin{equation}\label{LDP1compact}
J_K\le -\inf_{x\in K}I_{\rm min}(x)\quad\mbox{ for all }K\subset S\mbox{ compact.}
\end{equation}
\end{corollary}
\begin{proof}
Suppose that $A\subset S$ is open. 
For every $x\in A$ there exists $\delta_0>0$ such that $B_{\delta_0}(x)\subset A$. 
Therefore, by Proposition~\ref{prop:ratefunction} we have
\[
-I_{\rm min}(x)=\lim_{\delta\downarrow 0}J_{B_\delta(x)}\le J_{B_{\delta_0}(x)}\le  J_{A}.
\] 
By taking the supremum over all $x\in A$ we get \eqref{LDP1}.  

Suppose now that $\phi$ is max-stable and $K\subset S$ is compact.  
Given $\varepsilon>0$, due to Proposition \ref{prop:ratefunction} and by compactness there exist $x_1,\ldots,x_N\in K$ and $\delta_1,\ldots,\delta_N>0$ such that $K\subset\bigcup_{i=1}^N B_{\delta_i}(x_i)$ and
\[
-J_{B_{\delta_i}(x_i)}\ge (I_{\rm min}(x_i)-\varepsilon)\wedge \varepsilon^{-1}\quad\mbox{ for }i=1,\ldots,N.
\]
Then, due to Lemma~\ref{lem:max-stableJ}, 
\[
-J_K\ge \wedge_{1\le i\le N} (-J_{B_{\delta_i}(x_i)}) \ge \wedge_{1\le i\le N} (I_{\rm min}(x_i)-\varepsilon)\wedge\varepsilon^{-1}\ge \left(\underset{x\in K}\inf I_{\rm min}(x) - \varepsilon\right)\wedge \varepsilon^{-1}. 
\]
Then, \eqref{LDP1compact} follows as $\varepsilon>0$ was arbitrary.
\end{proof}

\begin{definition}
We say that the monetary risk measure $\phi$ satisfies the large deviation principle (LDP) with rate function $I$, if 
\begin{equation}\label{LDP}
-\inf_{x\in {\rm int}(A)}I(x) \le J_{A}\le -\inf_{x\in {\rm cl}(A)}I(x)\quad \mbox{for all }A\in \mathcal{B}(S).
\end{equation}	
\end{definition}
By Corollary~\ref{cor:LDPlower}, the lower bound in~\eqref{LDP} is always satisfied for the minimal rate function $I_{\rm min}$, and if in addition 
$\phi$ is max-stable, then the upper bound in~\eqref{LDP} holds for the minimal rate function $I_{\rm min}$ whenever $A$ is relatively compact. 

The following result shows that the LDP uniquely determines the rate function $I$. Namely, whenever $\phi$ satisfies the LDP with rate function $I$, it necessarily holds $I=I_{\rm min}$.

\begin{proposition}\label{prop:uniquenessrate} 
If $J_{A}\le -\inf_{x\in A} {I}(x)$ for every closed set $A\subset S$, then  $I_{\rm min}\ge I$. Analogously, if $-\inf_{x\in A}{I}(x)\le J_{A}$ for every open set  $A\subset S$, then $I_{\rm min}\le I$.
In particular, if $\phi$ satisfies the LDP with rate function $I$, then $I=I_{\rm min}$.    
\end{proposition}
\begin{proof}
Suppose, for instance, that $J_{A}\le -\inf_{x\in A} I(x)$ for every $A\subset S$ closed. 
Fix $x_0\in S$ and for every $\delta>0$ consider the closed ball $\bar{B}_\delta(x_0):=\{x\in S\colon d(x_0,x)\le \delta\}$. 
Then, since $I$ is lower semicontinuous,  it holds 
$$I(x_0)=\underset{\delta \downarrow 0}\lim\underset{x\in B_\delta(x_0)}\inf I(x)=\underset{\delta \downarrow 0}\lim\underset{x\in \bar{B}_\delta(x_0)}\inf I(x)\le -\underset{\delta \downarrow 0}\lim J_{\bar{B}_\delta(x_0)}=-\underset{\delta \downarrow 0}\lim J_{{B}_\delta(x_0)}=I_{\rm min}(x_0),$$ 
where the last equality follows from Proposition~\ref{prop:ratefunction}. 
\end{proof}	

The following proposition shows that the LP in  \eqref{eq:representation} is a sufficient condition for the LDP, generalizing the well-known fact in large deviations due to Bryc~\citep{bryc1990large}. 
In particular, we have that $I_{\rm min}$ is not only the minimal rate function, but also the unique possible rate function to obtain a representation \eqref{eq:representation}.  
We would like to emphasize that, in the following proposition, $I$ is not assumed to have compact sublevel sets.  
The proof is an adaptation of the proof of~\citep[Theorem 4.4.13]{zeitouni1998large}, where it is shown   that the LP implies the LDP for random variables with values in a completely regular topological space. 
However, \citep[Theorem 4.4.13]{zeitouni1998large} assumes that the rate function has compact sublevel sets to deal with the technical difficulties of working with a general completely regular topological space.

\begin{proposition}\label{prop:LPimpliesLDP}
If $\phi\colon B_b(S)\to\mathbb{R}$ has the representation
\begin{equation}\label{rep2}
\phi(f)=\sup_{x\in S}\{f(x)-I(x)\}
\end{equation}
for all $f\in C_b(S)$, then $\phi$ satisfies the LDP with rate function~$I$ and $I=I_{\rm min}$.
\end{proposition}
\begin{proof}
Fix a closed set $A\subset S$.  
We prove that $J_A\le-\inf_{x\in A} I(x)$ for every closed set $A\subset S$. 
If $\inf_{x\in A} I(x)=0$, then the assertion holds as $J_A\le 0$ by definition.  
Thus, assume that $\inf_{x\in A} I(x)>0$.  
	Let $\delta>0$ be with $\delta<\inf_{x\in A} I(x)$	and define \[I^\delta(x):=(I(x)-\delta)\wedge \delta^{-1}\quad\mbox{for }x\in S.\]  
	Set  $\alpha:=\inf_{x\in A}I^\delta(x)\in(0,+\infty)$. 
	Since $-\inf_{x\in S} I(x)=\phi(0)=0$, it follows 
	that $\{I\le \alpha\}$ is non-empty and closed as $I$ is lower semicontinuous. 
	Moreover, it holds $\{I\le\alpha\}\cap A=\emptyset$. 
	Due to Urysohn's lemma, for every $m\in\mathbb{N}$, there exists a continuous function $h_m\colon S\to [0,m]$ such that $\{I\le \alpha\}\subset h_m^{-1}(m)$ and $A\subset h_m^{-1}(0)$. 
	Then, it follows
	\[
	J_{A}=\inf_{r\in\mathbb{R}}\phi(r 1_{A^c})\le \phi(-h_m)=-\inf_{x\in S}\{h_m(x)+I(x)\}.
	\]
	Since $h_m(x)+I(x)\ge m$ if $x\in \{I\le \alpha\}$, and $h_m(x)+I(x)\ge \alpha$ if $x\notin \{I\le \alpha\}$, by choosing  $m\ge \alpha$ it follows that
	\[
	J_{A}\le -\alpha=-\inf_{x\in A}I^\delta(x).
	\] 
	Since the inequality above is satisfied for every $\delta>0$ small enough, we have 
	\[J_A\le-\lim_{\delta\downarrow 0}\inf_{x\in A} I^\delta(x)
	=-\lim_{\delta\downarrow 0}\Big((\inf_{x\in A} I(x) - \delta)\wedge \delta^{-1}\Big)
	=-\inf_{x\in A} I(x).\] 
	 
	It follows from~\eqref{rep2} that $I\ge I_{\rm min}$ due to Remark~\ref{rem:Imin}. 
	On the other hand, since $I$ satisfies the upper bound of the LDP, we have that $I\le I_{\rm min}$ as a consequence of Proposition~\ref{prop:uniquenessrate}.  
	This proves that $I=I_{\rm min}$.

Finally, since $I=I_{\rm min}$, it follows from Corollary~\ref{cor:LDPlower}, that $I$ satisfies the lower bound in the LDP.  	
The proof is complete.
\end{proof}

We next establish the equivalence between LP and LDP for max-stable monetary risk measures.  
Namely, we show that the converse of Proposition~\ref{prop:LPimpliesLDP} holds true under max-stability, which gives, in more generality, an alternative proof  for the classical Varadhan's lemma, see e.g.~\citep[Theorem III.13, p. 32]{den2008large}. 

 \begin{theorem}\label{thm:maxStVaradhan}
	Suppose that $\phi$ is max-stable.  
	Then, $\phi$ satisfies the LDP with rate function $I$ if and only if
	\begin{equation}\label{rep1}
	\phi(f)=\sup_{x\in S}\{f(x)-I(x)\}\quad\mbox{ for all }f\in C_b(S).
	\end{equation}
	In that case, it holds $I=I_{\rm min}$.
\end{theorem}

\begin{proof}
We know from Proposition~\ref{prop:LPimpliesLDP} that the representation \eqref{rep1} implies the LDP and that $I=I_{\rm min}$. 	 
Suppose that $\phi$ satisfies the LDP with rate function $I$ and fix $f\in C_b(S)$. 
Since $\phi$ is max-stable, it follows from Proposition~\ref{prop:charactMaxSt} that 
\begin{equation}\label{eq:integralRep}
\phi(f)=\underset{r\in\R}\sup\{r + J_{\{f>r\}}\}.
\end{equation}
 Given $r\in\R$, due to the upper bound of the LDP, it holds $J_{\{f\ge r\}}\le -\inf_{x\in \{f\ge r\}} I(x)$ as $\{f\ge r\}$ is closed. Then, by \eqref{eq:integralRep} and monotonicity of the concentration function $J$, we have
\begin{align*}
\phi(f)&\le \underset{r\in\R}\sup\big\{r -\inf_{x\in \{f\ge r\}} I(x)\big\}\\
&=\underset{r\in\R}\sup\;\underset{x\in \{f\ge r\}}\sup\{r -I(x)\}\\
&\le\underset{r\in\R}\sup\;\underset{x\in \{f\ge r\}}\sup\{f(x) -I(x)\}\\
&\le \underset{x\in S}\sup\{f(x) -I(x)\}.
\end{align*}

On the other hand, given $r\in\R$, it holds $J_{\{f>x\}}\ge -\inf_{x\in \{f>r\}} I(x)$ due to the lower bound of the LDP. 
Then,  given $\varepsilon>0$, it follows from \eqref{eq:integralRep} that
\begin{align*}
\phi(f)&\ge \underset{r\in\R}\sup\big\{r -\inf_{x\in \{f>r\}} I(x)\big\}\\
&=\underset{r\in\R}\sup\;\underset{x\in \{f> r\}}\sup\{r -I(x)\}\\
&\ge \underset{x\in S}\sup\;\underset{y\in \{f>f(x)-\varepsilon\}}\sup\{f(x)-\varepsilon -I(y)\}\\
&\ge \underset{x\in S}\sup\{f(x)-\varepsilon -I(x)\}.
\end{align*}
Since $\varepsilon>0$ was arbitrary, we obtain $\phi(f)\ge \underset{x\in S}\sup\{f(x) -I(x)\}$. The proof is complete.
\end{proof}

\begin{remark} 
In the particular case of the asymptotic entropy,
the LDP and LP are formulated in \citep{puhalskii2001large} by means of idempotent probability measures and the Shilkret integral; see Definition 3.1.1 and Theorem 3.1.1 in \citep{puhalskii2001large}.\footnote{We thank an anonymous referee for pointing out this reference.} 
For a comparison with the present setting, we next adapt these formulations to general max-stable monetary risk measures by using the equivalent language of maxitive penalties and maxitive integrals. Actually, $\mu$ is a maxitive penalty if and only if $e^\mu$ is a idempotent probability, and the maxitive integral~\eqref{eq:maxitiveInt} can be obtained as a transformation of the Shilkret integral; see~\citep{cattaneo} for more details.  
Given a rate function $I$, the function $A\mapsto I_A\colon \mathcal{B}(S)\to[-\infty,0]$ defined as
\[I_A:=-\underset{x\in A}\inf I(x)\] is a max-stable penalty. 
Then, the LDP can be reformulated in terms of penalties as follows.   
Namely, $\phi\colon B_b(S)\to\R$ satisfies the LDP with rate function $I$ if and only if  
\begin{equation}\label{capLDP}
I_{{\rm int}(A)}\le J_A\le I_{{\rm cl}(A)}\quad\mbox{ for all }A\in\mathcal{B}(S).
\end{equation}
Moreover, the LP principle can be reformulated as an equality of maxitive integrals. 
Namely,  $\phi$ satisfies the LP with rate function $I$ if and only if
\begin{equation}\label{capLP}
\int^X f {\rm d}J=\int^X f {\rm d}I\quad\mbox{ for all }f\in C_b(S).
\end{equation}
The latter equivalence can be proved by using the maxitive integral representation in Proposition~\ref{prop:charactMaxSt} together with a similar argument as in the proof of the ``only if'' part of Theorem~\ref{thm:maxStVaradhan}. Within this framework, it shoud be noticed that Theorem~\ref{thm:maxStVaradhan} states the equivalence between \eqref{capLDP} and \eqref{capLP}.
\end{remark}

We finish by discussing the scope of the present section. For the sake of simplicity, we stated our results for a metric space $S$.  
However, the statements are valid for normal spaces.\footnote{Recall that $S$ is normal if for any two disjoint closed subsets $A,B$ of $S$ there exist two disjoint open sets $U,V$ such that $A\subset U$ and $B\subset V$.} Moreover, all the results except Propositions~\ref{prop:LPimpliesLDP} (and consequently the ``if'' part of  Theorem \ref{thm:maxStVaradhan}) are valid for \emph{completely regular spaces}.\footnote{Recall that a topological space $S$ is completely regular if it is Hausdorff and for any closed set $A\subset S$ and any point $x\notin A$, there exists a continuous function $f\colon S\to [0,1]$ such that $f(x)=1$ and $A\subset f^{-1}(0)$.} 
Propositions~\ref{prop:LPimpliesLDP} is also valid for completely regular topological spaces if one assumes that the rate function has compact sublevel sets. 
The argument follows from an adaptation to the present setting of the proof of~\citep[Theorem 4.4.13]{zeitouni1998large}.  
Finally, the ``only if'' part of  Theorem~\ref{thm:maxStVaradhan} 
(and consequently Varadhan's lemma) is valid for general topological spaces as it is a consequence of Proposition~\ref{prop:charactMaxSt}. 

\section{Locally max-stable monetary risk measures}\label{sec:locmax-stable} 
Throughout this section we consider a monetary risk measure  $\phi\colon B_b(S)\to\mathbb{R}$. 
Given a compact set $K\subset S$, we denote by $C^K(S)$ the space of all real-valued functions from $S$ to $\mathbb{R}$ of the form $f 1_K+r 1_{K^c}$, where $f\in C(K)$ and $r\in\mathbb{R}$. 

\begin{definition}
We say that $\phi$ is \emph{locally max-stable} if $\phi\mid_{C^K(S)}$ is max-stable for every compact set $K\subset S$.
\end{definition}

\begin{lemma}\label{lem2}
	Suppose that $\phi$ is locally max-stable.  
	Then, for every compact set $K\subset S$, $\phi\mid_{C^K(S)}$ has the representation 
	\[
	\phi(f 1_K  + r 1_{K^c} )=\sup_{x\in K}\{f(x) - I_K(x)\}\vee (r+J_{K^c})
	\]
	for all $f\in C(K)$ and $r\in\mathbb{R}$, where
	\[
	I_K(x):=\sup_{f\in C(K)}\Big\{f(x)-\underset{r\in\mathbb{R}}\inf\phi(f 1_K + r 1_{K^c})\Big\}.
	\] 
	Recall that $J_{K^c}=\inf_{r\in \mathbb{R}}\phi(r 1_{K})$.
\end{lemma}
\begin{proof}
 Fix a compact set $K\subset S$. 
 If $S=K$, then the result follows from Corollary~\ref{cor:compact}.     
 Suppose that $K\neq S$ and choose $x_0\in S\setminus K$. 
 Define 
 \[
 \phi_K\colon C(K\cup\{x_0\})\to\R,\quad \phi_K(f):=\phi(f 1_K + f(x_0) 1_{K^c}).
 \] 
 Since $\phi\mid_{C^K(S)}$ is max-stable, it follows that $\phi_K$ is max-stable. 
 We have that $K\cup\{x_0\}$ is compact, then by applying Corollary~\ref{cor:compact}  to $\phi_K$ we get
    \[
	\phi(f 1_K  + r 1_{K^c})=\phi_K(f 1_K  + r 1_{\{x_0\}})=\sup_{x\in K}\{f(x)- \phi^\ast_K(\delta_x)\}\vee
	(r- \phi_K^\ast(\delta_{x_0})).
	\]
	If $x\in K$, then it holds
	\begin{align*}
	\phi_K^\ast(\delta_x)=&\sup_{f\in C(K\cup\{x_0\})}\{ f(x) - \phi(f 1_K  + f(x_0) 1_{K^c})\}\\
	=&\sup_{f\in C(K),\,r\in \mathbb{R}}\{ f(x) - \phi(f 1_K  + r 1_{K^c})\}\\
	=&\sup_{f\in C(K)}\{ f(x) - \inf_{r\in \mathbb{R}}\phi(f 1_K  + r 1_{K^c})\}\\
	=&I_K(x).
	\end{align*}
	On the other hand, 
	\begin{align*}
	\phi_K^\ast(\delta_{x_0})=&\sup_{f\in C(K\cup\{x_0\})}\{ f(x_0) - \phi(f 1_K  + f(x_0) 1_{K^c})\}\\
	=&\sup_{f\in C(K),\,r\in \mathbb{R}}\{ r - \phi(f 1_K  + r 1_{K^c})\}\\
	=&\sup_{f\in C(K),\,r\in \mathbb{R}}\{ - \phi(f 1_K  + r 1_{K^c}-r)\}\\
	=&\sup_{f\in C(K),\,r\in \mathbb{R}}\{ - \phi(f 1_K  - r 1_{K})\}\\
	=&\sup_{f\in C(K),\,r\in \mathbb{R}}\{ - \phi((f - r)1_K )\}\\
	=&-\inf_{r\in\mathbb{R}}\phi(r 1_K)=-J_{K^c}.
	\end{align*}
	The proof is complete.
\end{proof}

\begin{lemma}\label{lem:IK}
For every compact set $K\subset S$, it holds
\begin{itemize}
	\item[(i)] $I_K(x)\ge I_{\rm min}(x)$ for all $x\in K$,
	\item[(ii)] $I_{\rm min}(x)\ge -J_{K^c}$ for all $x\notin K$. 
\end{itemize}
\end{lemma}	
\begin{proof}
As for (i), suppose that $K\subset S$ is compact and fix $x\in K$.   
Given $f\in C_b(S)$, we have \[f(x) - \phi(f1_K  - r 1_{K^c})\ge f(x)-\phi(f),\]
for all $r\in\mathbb{R}$ with $r\ge \|f\|_\infty$, which shows that 
\[f(x) - \inf_{r\in\mathbb{R}}\phi(f1_K  - r 1_{K^c})\ge f(x)-\phi(f).\]
Since by Tietze's extension theorem, every $f\in C(K)$ has an extension $\bar f\in C_b(S)$ such that $\bar f\mid_K=f$, we conclude
\[I_K(x)=\sup_{f\in C(K)}\big\{f(x) - \inf_{r\in\mathbb{R}}\phi(f1_K  - r 1_{K^c})\big\}\ge \sup_{f\in C_b(S)}\{f(x)-\phi(f)\}=I_{\rm min}(x).\]
Finally, notice that (ii) is a direct consequence of Corollary~\ref{cor:LDPlower}.
\end{proof}	


\begin{theorem}\label{th:localMaxSt}
	Suppose that $\phi$ is locally max-stable and one of the following conditions is satisfied:
	\begin{itemize}
		\item[(A)] For every $M>0$, there exists a compact set $K\subset S$ such that $-J_{K^c}\ge M$.
		
		\item[(B)] There exists $I(\infty)\in \mathbb{R}$ such that for every $\varepsilon>0$  there exists a compact set $K\subset S$ so that
		\[-J_{K^c}\ge I(\infty)-\varepsilon \quad\mbox{and}\quad I_{\rm min}(x)\le I(\infty)+\varepsilon\mbox{ for all }x\in K^c.\] 
	\end{itemize}
	Then, it holds
	\begin{equation}\label{rep:main:local}
	\phi(f)=\sup_{x\in S}\{f(x)-I_{\rm min}(x)\}\quad \mbox{for all }f\in C_b(S)
	\end{equation}
	and $\phi$ satisfies the LDP with rate function $I_{\rm min}$. 
	Moreover, if (A) holds, then $I_{\rm min}$ has compact sublevel sets.
\end{theorem}

\begin{proof} 
Fix $f\in C_b(S)$ and set $a:=\sup_{x\in S}\{ f(x) - I_{\rm min}(x)\}$. 
By definition of $I_{\rm min}$ one has that $\phi(f)\ge a$. 
Thus, we have to prove that $\phi(f)\le a$.

We first assume that condition (A) holds.
Fix $M>0$ and choose a compact set $K\subset S$ such that $-J_{K^c}\ge M$.  
Define $\bar{f}:=f 1_K  +  \Vert f\Vert_\infty 1_{K^c}$. 
By Lemma \ref{lem2} and Lemma \ref{lem:IK} we obtain 
\begin{align*}
\phi(f)&\le \phi(\bar{f})=\sup_{x\in K} \{f(x) - I_K(x)\} \vee  \{\Vert f\Vert_\infty + J_{K^c} \}\le a \vee \{\Vert f\Vert_\infty - M\}.
\end{align*}
By choosing $M$ large enough, we obtain that $\phi(f)\le a$, which shows \eqref{rep:main:local}. 

We prove that $I_{\rm min}$ has compact sublevel sets. 
Given $r\ge 0$, there exists a compact set $K\subset S$ such that
\[
-J_{K^c}\ge r+1.
\]
Then we have $I^{-1}_{\rm min}([0,r])\subset K$, since otherwise
there exists $x\in I^{-1}_{\rm min}([0,r])$ such that $x\notin K$, which by Lemma \ref{lem:IK} would imply $r\ge I_{\rm min}(x)\ge -J_{K^c} \ge r+1$. Since $I_{\rm min}$ is lower semicontinuous, it follows that $I^{-1}_{\rm min}([0,r])$ is compact.

Now, suppose that condition (B) holds. 
Fix $\varepsilon>0$ and choose a compact set $K\subset S$ such that $I_{\rm min}(x)\le I(\infty)+\varepsilon$ for all $x\in K^c$, and $-J_{K^c}\ge I(\infty)-\varepsilon$. Define $\tilde{f}:=f 1_K  + (a + I(\infty) + \varepsilon)1_{K^c}$. Then, since $a\ge f(x)-I_{\rm min}(x)$ for all $x\in S$, we have 
\[
f=f 1_K  + f 1_{K^c} \le f 1_K  +  (a + I_{\rm min}) 1_{K^c}\le f 1_K  + (a + I(\infty)+\varepsilon)1_{K^c} = \tilde{f}.
\]
Thus, by Lemma \ref{lem2} and Lemma \ref{lem:IK}, we obtain 
\begin{align*}
\phi(f)&\le \phi(\tilde{f})\\
&=\sup_{x\in K} \{f(x) - I_K(x)\} \vee \{a + I(\infty)+\varepsilon + J_{K^c} \}\\
&\le a \vee  \{a + I(\infty)+\varepsilon - I(\infty)+\varepsilon\}\\  
&\le a \vee (a +2\varepsilon)\le \sup_{x\in S}\{f(x)-I_{\rm min}(x)\} + 2\varepsilon.
\end{align*}
As $\varepsilon$ was arbitrary, we derive \eqref{rep:main:local}.

Finally,  since $\phi$ has the representation \eqref{rep:main:local}, by  Proposition~\ref{prop:LPimpliesLDP},  $\phi$ satisfies the LDP with rate function $I_{\rm min}$ . 
The proof is complete.\end{proof}

\begin{remark}
Suppose that $\hat{S}:=S\cup\{\infty\}$ is the Alexandroff one-point compactification of $S$.  
Recall that $\hat{S}$ is endowed with the topology consisting of all open subsets of $S$ together with all sets of the form $V=K^c\cup\{\infty\}$, where $K$ is compact; see \citep{kelley} for further details. 
Then, conditions (A) and (B) in Theorem~\ref{th:localMaxSt} can be unified
as a continuity condition on $\hat{S}$. 
Namely, extend $I_{\rm min}$ to a function $\hat{I}\colon \hat{S}\to[0,+\infty]$ by setting $\hat{I}(\infty):=\sup\{-J_{K^c}\colon K\subset S\text{ compact}\}$, and consider $[0,+\infty]$ as the one-point compactification of $[0,+\infty)$.  
Then, it can be verified that $\hat{I}$ is continuous at $\infty$ if and only if one of the conditions (A) or (B) is satisfied. 
Further, condition (A) corresponds to the case $\hat{I}(\infty)=+\infty$, and condition (B) corresponds to the case $\hat{I}(\infty)<+\infty$. 

Under condition (A) the rate function $I_{\rm min}$ has compact sublevel sets.   
However, the following example shows that the compactness of the sublevel sets of $I_{\rm min}$ is not a sufficient condition for the representation~\eqref{rep:main:local}.  
\end{remark}

\begin{example}
Suppose that $(X_n)_{n\in\N}$ is a sequence of random variables with values in $\mathbb{N}$,  
such that for every $m\in\mathbb{N}$, it holds 
\[
\mathbb{P}(X_n=m)=
\begin{cases}
e^{-n m}, &\mbox{ if }n\neq m,\\
1-\frac{e^{-n}}{1-e^{-n}}+e^{-n^2}, & \mbox{ if }n=m.
\end{cases}
\]
Then, for every $f=(f(m))_{m\in\mathbb{N}}\in B_b(\mathbb{N})$, the upper asymptotic entropy is given by
\[
\overline{\phi}_{\rm ent}(f):=\limsup_{n\to\infty}\frac{1}{n}\log\mathbb{E}[\exp(n f(X_n))]=\limsup_{n\to\infty}\frac{1}{n}\log\sum_{m=1}^\infty e^{n f(m)} \mathbb{P}(X_n=m).
\]
In this case, the rate function is explicitly given (see~\eqref{eq:rateexp} below) by
\begin{align*}
I_{\rm min}(m)&=-\lim_{\delta\downarrow 0}\limsup_{n\to\infty}\frac{1}{n}\log\mathbb{P}(X_n\in B_\delta(m))\\
&=-\limsup_{n\to\infty}\frac{1}{n}\log\mathbb{P}(X_n=m)\\
&=-\limsup_{n\to\infty}\frac{1}{n}\log e^{-n m}=m,
\end{align*}
so that $I_{\rm min}$ has compact sublevel sets.  

Moreover, defining $K_m:=\{1,2,\ldots,m\}$ it follows from \eqref{eq:Jexp} below that 
\begin{align*}
J_{K^c_m}
&=\limsup_{n\to\infty}\frac{1}{n}\log\mathbb{P}(X_n> m)\\
&=\limsup_{n\to\infty}\frac{1}{n}\log \left(\sum_{k=m+1}^\infty e^{-n k} + 1-\frac{e^{-n}}{1-e^{-n}}\right)\\
&= \underset{n\to\infty}\lim\frac{1}{n}\log \left(\frac{e^{-n(m+1)}}{{1-e^{-n}}}+ 1-\frac{e^{-n}}{1-e^{-n}}\right)=0.
\end{align*}
Since every compact $K\subset \mathbb{N}$ is finite, it holds $K\subset K_m$ for $m$ large enough, and therefore $0=J_{K_m^c}\le J_{K^c}$. 
Hence, as the concentration function $J$ assumes values in $[-\infty,0]$, it follows that $J_{K^c}=0$. 
We conclude that  condition (B) is not satisfied as $I_{\rm min}$ is not bounded on the complement of a compact subset, and condition (A) does not hold as $\hat{I}(\infty)=0<+\infty$.
 
Finally, we show that $\overline{\phi}_{\rm ent}$ does not admit the representation~\eqref{rep:main:local}. 
Indeed, for the constant function $f\equiv 1$, it holds 
$\sup_{m\in\mathbb{N}}\{f_m - I_{\rm min}(m)\}=\sup_{m\in\mathbb{N}}\{1 - m\}=0$, but
\[
\overline{\phi}_{\rm ent}(f)=\limsup_{n\to\infty}\frac{1}{n}\log\mathbb{E}[\exp(n f(X_n))]=\limsup_{n\to\infty}\frac{1}{n}\log\mathbb{E}[\exp(n)]=1.
\]
As a consequence, the compactness of the sublevel sets of $I_{\rm min}$ is not a sufficient condition for  the representation~\eqref{rep:main:local}.
\end{example}

We next provide an example where condition (A) does not hold, but condition (B) is satisfied. 

\begin{example}
 Suppose that $\mathbb{Q}=\{q_1,q_2,\ldots\}$ is an enumeration of the set of rational numbers. 
Consider the deterministic random variables $X_n\equiv q_n$, for $n\in\mathbb{N}$. 
Then, for every $f\in C_b(\mathbb{R})$ the asymptotic entropy is given by
\[
\overline{\phi}_{\rm ent}(f)=\limsup_{n\to\infty}\frac{1}{n}\log\mathbb{E}[\exp(n f(X_n))]=\sup_{r\in\mathbb R}f(r).
\]
 By \eqref{eq:rateexp} we have
\[
I_{\rm min}(x)=-\lim_{\delta\downarrow 0}\limsup_{n\to\infty}\frac{1}{n}\log\mathbb{P}(X_n\in B_\delta(x))=0
\]
for every $x\in\mathbb{R}$. 
Further, by~\eqref{eq:Jexp} for every compact set $K\subset \mathbb{R}$, it holds
\[
J_{K^c}=\lim_{\delta\downarrow 0}\limsup_{n\to\infty}\frac{1}{n}\log\mathbb{P}(X_n\in K^c)=0.
\]
Then, $\overline{\phi}_{\rm ent}$ satisfies condition (B) for $I(\infty):=0$, but it does not satisfy  condition (A).
\end{example}

\begin{remark}
Let us discuss the scope of Theorem~\ref{th:localMaxSt}. Under the assumption that 
$\phi\colon B_b(S)\to\mathbb{R}$ is max-stable, 
Theorem~\ref{th:localMaxSt} is valid when $S$ is a completely regular topological space. 
In that case, each of the conditions (A) and (B) implies the LDP, and therefore the representation~\eqref{rep:main:local} follows from Theorem~\ref{thm:maxStVaradhan}. 

For instance, suppose that (B) is satisfied and $\phi$ is max-stable.   
Then, for $A\subset S$ closed and $\varepsilon>0$, there exists a compact set $K\subset S$ such that $I_{\rm min}(x)\le I(\infty)+\varepsilon$ for all $x\in K^c$, and $-J_{K^c}\ge I(\infty)-\varepsilon$. 
Then, by Corollary~\ref{cor:LDPlower} and Lemma~\ref{lem:max-stableJ}, since $A\cap K$ is compact, it holds 
\begin{align*}
J_{A}\le& J_{A\cap K}\vee J_{K^c}\\
&\le -\underset{x\in A\cap K}\inf I_{\rm min}(x)\vee \left(-I(\infty)+\varepsilon\right)\\
&\le -\underset{x\in A\cap K}\inf I_{\rm min}(x)\vee \big(-\underset{x\in A\cap K^c}\inf I_{\rm min}(x)+ 2\varepsilon\big)\\
&\le -\underset{x\in A}\inf I_{\rm min}(x)+ 2\varepsilon.
\end{align*}
Since $\varepsilon>0$ was arbitrary, we get $J_A\le -\inf_{x\in A} I(x)$. A similar argument shows that condition (A) together with the max-stability of $\phi$ implies the LDP.  
\end{remark}

\section{Large deviations for pairs of monetary risk measures}\label{sec:pairs}
For the forthcoming discussion, let $\overline{\phi},\,\underline{\phi}\,\colon B_b(S)\to\mathbb{R}$   
be  two monetary risk measures such that 
\[\underline{\phi}(f)\le \overline{\phi}(f)\quad\mbox{ for all }f\in C_b(S).\]
Let $I\colon S\to[0,+\infty]$ be a rate function.
We will provide conditions under which $\overline{\phi}$ and $\underline{\phi}$ coincide on $C_b(S)$, and satisfy a LP with rate function $I$. 
 Define the respective minimal rate functions $\overline{I}$ and $\underline{I}$ as in \eqref{eq:Imin}, and the associated concentration functions $\overline{J}$ and $\underline{J}$ as in \eqref{eq:JA}. It follows directly from the definitions that $\overline{I}(x)\le\underline{I}(x)$  for all $x\in S$, and $\underline{J}_A\le \overline{J}_A$ for all $A\in\mathcal{B}(S)$. 

\begin{proposition}\label{prop:upperlower}
Suppose that $\overline{\phi}$ satisfies the LP with rate function $I$, as well as $-\inf_{x\in A}I(x)\le \underline{J}_{A}$ for every open set $A\subset S$. Then, $\overline{\phi}(f)=\underline{\phi}(f)$ for all $f\in C_b(S)$ and $\underline{I}=\overline{I}=I$. 
 In that case, it holds
 \begin{equation}\label{eq:extendedLDP0}
-\inf_{x\in \textnormal{int} (A)}{I}(x)\le \underline{J}_{A}\le \overline{J}_A\le -\inf_{x\in \textnormal{cl} (A)}{I}(x)
\end{equation}
for all $A\in\mathcal{B}(S)$.
\end{proposition}
\begin{proof}
Since $-\inf_{x\in A}I(x)\le \underline{J}_{A}$ 
 for all $A\subset S$ open, by applying Proposition~\ref{prop:uniquenessrate} to $\underline{\phi}$, we have that $\underline{I}\le {I}$. 
 Therefore, for $f\in C_b(S)$, it holds
 \[
\overline{\phi}(f)\ge\underline{\phi}(f)\ge \sup_{x\in S}\{f(x)-\underline{I}(x)\}\ge \underset{x\in S}\sup\{f(x)-{I}(x)\}= \overline{\phi}(f).
 \]
We get that $\overline{\phi}(f)=\underline{\phi}(f)$ for all $f\in C_b(S)$. 
Finally, from Proposition~\ref{prop:LPimpliesLDP}, it follows that $\underline{I}=\overline{I}=I$ and \eqref{eq:extendedLDP0} holds. 
\end{proof}

\begin{proposition}\label{prop:extendedLDP}
Suppose that $\overline{\phi}$ is max-stable.   
Then, 
\begin{equation}\label{eq:extendedLDP}
-\inf_{x\in \textnormal{int} (A)}{I}(x)\le \underline{J}_{A}\le \overline{J}_A\le -\inf_{x\in \textnormal{cl} (A)}{I}(x)\quad\mbox{for all } A\in\mathcal{B}(S)
\end{equation}
 if and only if
\begin{equation}\label{eq:representationLower}
\overline{\phi}(f)=\underline{\phi}(f)=\sup_{x\in S}\{f(x)-{I}(x)\}\quad\mbox{for all }f\in C_b(S). 
\end{equation}
In that case, it holds ${I}=\underline{I}=\overline{I}$.
\end{proposition} 
\begin{proof}

Suppose that \eqref{eq:extendedLDP} holds. 
In particular, $\overline{\phi}$ satisfies the LDP with rate function $I$ and consequently  $\overline{\phi}$ has the representation \eqref{eq:representationLower} due to Theorem~\ref{thm:maxStVaradhan}.  
Moreover, the lower bound in \eqref{eq:extendedLDP} implies that $\overline{\phi}(f)=\underline{\phi}(f)$  for all $f\in C_b(S)$ due to Proposition~\ref{prop:upperlower}. Hence, $\underline{\phi}$ also has the representation \eqref{eq:representationLower}. 
Conversely, suppose that \eqref{eq:representationLower} holds. 
By Theorem~\ref{thm:maxStVaradhan}, $\overline{\phi}$ and $\underline{\phi}$ satisfy the LDP with rate function $I$, which shows \eqref{eq:extendedLDP}. 
\end{proof}

 Next, we state a version of Theorem~\ref{th:localMaxSt} in the present context.
 
\begin{proposition}\label{cor:UpperLowerBryc}
Suppose that $\overline{\phi}$ and $\underline{\phi}$ are locally max-stable. 
If $\overline{\phi}$ satisfies one of the condition (A) or (B), then $\overline\phi$ satisfies the LP with rate function $\overline{I}$, $\underline\phi$ satisfies the LP with rate function $\underline{I}$,
and
\begin{equation}\label{eq:LDPuplo}
-\underset{x\in {\rm int}(A)}\inf \underline{I}(x)
\le \underline{J}_A\le \overline{J}_A
\le -\underset{x\in {\rm cl}(A)}\inf \overline{I}(x)\quad\mbox{for all }A\in\mathcal{B}(S).
\end{equation}
If, in addition, $-\underset{x\in A}\inf \overline{I}(x)
\le \underline{J}_A$ for every open set $A\subset S$, then $\underline{I}=\overline{I}$ and $\underline{\phi}(f)=\overline{\phi}(f)$ for all $f\in C_b(S)$.	
\end{proposition} 
\begin{proof}
Since $\overline{\phi}$ satisfies one of the condition (A) or (B), it follows that $\underline{\phi}$ also satisfies one of these conditions by just noting that $\overline{I}\le \underline{I}$ and $\underline{J}\le \overline{J}$.  Hence, the first part follows directly from Theorem \ref{th:localMaxSt}. Finally, the second part is an application of Proposition~\ref{prop:upperlower}. The proof is complete.
\end{proof}

\section{Asymptotic shortfall risk measures}\label{sec:examples}

Given a non-decreasing loss function $l\colon\mathbb{R}\to\mathbb{R}$ the \emph{shortfall risk} of a bounded random variable $Z\colon\Omega\to\mathbb{R}$ is defined as
\[
\phi_l(Z):=\inf\big\{m\in\mathbb{R}\colon \mathbb{E}[l(Z-m)]\le 1\big\}.
\]
For further details on shortfall risk measures we refer to \citep{foellmer01}. 
In the following we consider a sequence $(X_n)_{n\in\mathbb{N}}$ of random variables $X_n\colon\Omega\to S$. Moreover, let
$(l_n)_{n\in\mathbb{N}}$ be a sequence of functions of the form
$l_n=\exp(w_n)$, where $w_n\colon (-\infty,+\infty]\to(-\infty,+\infty]$  is a non-decreasing function 
with $w_n(0)=0$ for all $n\in\mathbb{N}$. We work with the generalized inverse $w^{-1}_n\colon (-\infty,+\infty]\to [-\infty,+\infty]$, given by $w^{-1}_n(y):=\sup\{x\in\mathbb{R} \colon w_n(x)\le y\}$, with the convention $\sup\emptyset:=-\infty$. 
In addition, we assume that
\begin{equation}
\label{eq:asymptotic}
\limsup_{n\to\infty} w^{-1}_n(a+a_n)=\limsup_{n\to\infty} w^{-1}_n(a_n),\quad
\liminf_{n\to\infty} w^{-1}_n(a+a_n)=\liminf_{n\to\infty} w^{-1}_n(a_n)
\end{equation}
for all $a\in\mathbb{R}$ and every sequence $(a_n)_{n\in\mathbb{N}}$ in $[0,+\infty]$.

\begin{definition}\label{def:asymptoticshortfall}
The upper/lower \emph{asymptotic shortfall risk measures} $\overline{\phi},\,\underline{\phi}\colon B_b(S)\to\mathbb{R}$ are defined as
\[
\overline{\phi}(f):=\limsup_{n\to\infty} \phi_{l_n}\big( f(X_n)\big)= \limsup_{n\to\infty}\, \inf\big\{m\in\mathbb{R}\colon \mathbb{E}\big[l_n\big(f(X_n)-m\big)\big]\le 1\big\},
\]
\[
\underline{\phi}(f):=\liminf_{n\to\infty} \phi_{l_n}\big( f(X_n)\big)= \liminf_{n\to\infty}\, \inf\big\{m\in\mathbb{R}\colon \mathbb{E}\big[l_n\big(f(X_n)-m\big)\big]\le 1\big\}.
\]  
Straightforward inspection shows that $\overline{\phi},\,\underline{\phi}\colon B_b(S)\to\mathbb{R}$ are monetary risk measures. 
Define the respective rate functions $\overline{I}$, $\underline{I}$ as in \eqref{eq:Imin}, and the associated concentration functions $\overline{J}$, $\underline{J}$ as in  \eqref{eq:JA}, respectively.
\end{definition}

\begin{proposition}\label{prop:wrate}
	For every $B\in\mathcal{B}(S)$ and $x\in S$, it holds
	\[
	\overline{J}_B=-\liminf_{n\to\infty} w^{-1}_n\left(-\log\mathbb{P}(X_n\in B)\right)\quad\mbox{and}\quad
	\overline{I}(x)=\lim_{\delta\downarrow 0}\liminf_{n\to\infty} w^{-1}_n\left(-\log\mathbb{P}(X_n\in B_\delta(x))\right),
	\]
	\[
	\underline{J}_B=-\limsup_{n\to\infty} w^{-1}_n\left(-\log\mathbb{P}(X_n\in B)\right)\quad\mbox{and}\quad
	\underline{I}(x)=\lim_{\delta\downarrow 0}\limsup_{n\to\infty} w^{-1}_n\left(-\log\mathbb{P}(X_n\in B_\delta(x))\right).
	\]
\end{proposition}
\begin{proof}
    We show the result for $\overline{\phi}$,  
    the argumentation for $\underline{\phi}$ is similar.  
	For every $B\in\mathcal{B}(S)$ and  $r\in\mathbb{R}$, it holds 
	\begin{align*}
	\phi_{l_n}(r 1_{B^c})
=&\inf\Big\{m\in\R\colon\exp\big(w_n(r - m)\big) \mathbb{P}(X_n\in B^c) + \exp(w_n(- m)\big)\mathbb{P}(X_n\in B)\le 1\Big\}\\
\le &\inf\Big\{m\in\R\colon 2\big(\exp\big(w_n(r - m)\big) \mathbb{P}(X_n\in B^c) \vee \exp(w_n(- m)\big)\mathbb{P}(X_n\in B)\big)\le 1\Big\}\\
	=&\inf \left\{m\in\R\colon 
	\begin{array}{c}
	w_n(r - m)\le -\log 2 - \log \mathbb{P}(X_n\in B^c)\\
	w_n(- m)\le -\log 2 - \log \mathbb{P}(X_n\in B)
	\end{array}
	\right\}\\
	=&\inf \left\{m\in\R\colon 
	\begin{array}{c}
	m\ge  r- w_n^{-1}\big(-\log 2 - \log \mathbb{P}(X_n\in B^c)\big)\\
	m\ge -w_n^{-1}\big(-\log 2 - \log \mathbb{P}(X_n\in B)\big)
	\end{array}
	\right\}\\
	=&\Big( r- w_n^{-1}\big(-\log 2 - \log \mathbb{P}(X_n\in B^c)\big)\Big)\vee \Big(-w_n^{-1}\big(-\log 2 - \log \mathbb{P}(X_n \in B)\big)\Big).
	\end{align*}
	Using \eqref{eq:asymptotic}, we get
	\[
	\overline{\phi}(r 1_{B^c})\le\Big( r- \liminf_{n\to\infty}w_n^{-1}\big( - \log \mathbb{P}(X_n\in B^c)\big)\Big)\vee\Big(-\liminf_{n\to\infty}w_n^{-1}\big(- \log \mathbb{P}(X_n\in B)\big)\Big).
	\]
	By letting $r\to-\infty$, we obtain
	\[
	\overline{J}_B\le - \liminf_{n\to\infty} w^{-1}_n(-\log \mathbb{P}(X_n\in B)).
	\]  
	On the other hand, for every $r\in\mathbb{R}$, we have
	\begin{align*}
	\phi_{l_n}(r 1_{B^c})= & \inf\Big\{m\in\mathbb{R}\colon\exp\big(w_n(r -  m)\big) \mathbb{P}(X_n\in B^c) + \exp\big(w_n(- m)\big)\mathbb{P}(X_n\in B)\le 1\Big\}\\
	\ge & \inf\Big\{m\in\mathbb{R}\colon \exp\big(w_n(- m)\big)\mathbb{P}(X_n\in B)\le 1\Big\}\\
	=& -w^{-1}\Big(-\log \mathbb{P}(X_n\in B)\Big).
	\end{align*}
	From there, we get
	\[
	\overline{\phi}(r 1_{B^c})\ge -\liminf_{n\to\infty} \frac{1}{n}w^{-1}(-\log \mathbb{P}(X_n\in B)).
	\]
	By letting $r\to-\infty$, we obtain
	$\overline{J}_B\ge - \liminf_{n\to\infty} w^{-1}_n(-\log \mathbb{P}(X_n\in B))$.
	
	Finally, for $x\in S$, it follows from Proposition \ref{prop:ratefunction} that
	\[
	\overline{I}(x)=-\lim_{\delta\downarrow 0}\overline{J}_{B_\delta(x)}=\lim_{\delta\downarrow 0}\liminf_{n\to\infty} w^{-1}_n\big(-\log\mathbb{P}(X_n\in B_\delta(x))\big).
	\]
The proof is complete.	
\end{proof}

By Proposition \ref{prop:wrate}, it holds  $-\overline{J}_{K^c}=\liminf_{n\to\infty} w^{-1}_n\big(-\log\mathbb{P}(X_n\in K^c)\big)$ for every compact $K\subset S$. Hence, conditions (A) and (B) in Theorem \ref{th:localMaxSt} give rise to the following modified versions: 
\begin{itemize}
\item[(A')] For every $M>0$ there exists a compact set $K\subset S$ such that
	\[
\liminf_{n\to\infty} w^{-1}_n\big(-\log\mathbb{P}(X_n\in K^c)\big)\ge M. 
	\]
\item[(B')] There exists $\overline{I}(\infty)\in \mathbb{R}$ such that for every $\varepsilon>0$  there exists a compact set $K\subset S$ which satisfies
\[
\liminf_{n\to\infty} w^{-1}_n\big(-\log\mathbb{P}(X_n\in K^c)\big)\ge \overline{I}(\infty)-\varepsilon
\]
and
\[ 
\overline{I}(x)\le \overline{I}(\infty)+\varepsilon\quad\mbox{ for all }x\in K^c.
\] 
\end{itemize}

As a result, we obtain the following generalization of Bryc's lemma.
\begin{proposition}\label{prop:perturbedBryc}
	Suppose that $(X_n)_{n\in\mathbb{N}}$ satisfies one of the conditions (A'), (B'). 
	Then, for all $f\in C_b(S)$,
	\begin{equation}\label{eq:repShortFall}
		\limsup_{n\to\infty}\inf\Big\{m\in\mathbb{R}\colon \mathbb{E}\big[\exp\big(w_n(f(X_n)- m)\big)\big]\le 1 \Big\}=\sup_{x\in S}\big\{f(x)-\overline{I}(x)\big\}
	\end{equation}
	\begin{equation}\label{eq:repShortFallII}
		\liminf_{n\to\infty}\inf\Big\{m\in\mathbb{R}\colon\mathbb{E}\big[\exp\big(w_n(f(X_n)- m)\big)\big]\le 1 \Big\}=\sup_{x\in S}\big\{f(x)-\underline{I}(x)\big\}
	\end{equation}
 and for every $A\in \mathcal{B}(S)$, it holds
	\begin{align*}\label{eq:LDPShortFallI}
		-\underset{x\in{\rm int}(A)}\inf \underline{I}(x)&\le  -\limsup_{n\to\infty} w^{-1}_n\big(-\log\mathbb{P}(X_n\in A)\big)\\
		&\le -\liminf_{n\to\infty} w^{-1}_n\big(-\log\mathbb{P}(X_n\in A)\big)\le
		-\underset{x\in{\rm cl}(A)}\inf \overline{I}(x)
	\end{align*}
Moreover, if $\limsup_{n\to\infty} w^{-1}_n\big(-\log\mathbb{P}(X_n\in A)\big)\le \inf_{x\in A} \overline{I}(x)$ for every open set $A\subset S$, then $\overline{I}=\underline{I}$ and \eqref{eq:repShortFall}=\eqref{eq:repShortFallII}, in which case the $\limsup$/$\liminf$ is in fact a limit.
\end{proposition}

\begin{proof}
By assumption $\overline{\phi}$ satisfies  
either (A) or (B) in Proposition \ref{cor:UpperLowerBryc}. Hence, all assertions would follow from Proposition \ref{cor:UpperLowerBryc}, if both $\overline{\phi}$ and $\underline{\phi}$ were locally max-stable.  

 Suppose that $K\subset S$ is compact. We prove that $\overline{\phi}|_{C^K(S)}$ is max-stable, 
 the argumentation for $\underline{\phi}|_{C^K(S)}$ is similar. We first assume that $K=S$.  
    Fix $f\in C(S)$ and $\varepsilon>0$. 
	By compactness, there exist $x_1,\ldots,x_N\in K$ and $\delta_1,\ldots,\delta_N>0$ such that
	$K\subset \bigcup_{i=1}^N B_{\delta_i}(x_i)$,
	\[
	f(x)\le f(x_i)+\varepsilon\quad\mbox{ for all }x\in B_{\delta_i}(x_i) 
	\]
	and
	\[
	(\overline{I}(x_i)-\varepsilon)\wedge \varepsilon^{-1}\le -\overline{J}_{B_{\delta_i}(x_i)}.
	\] 
	Then, it holds
	\begin{align*}
	\phi_{l_n}(f)&\le \inf\Big\{m\in\R\colon \sum_{i=1}^N \exp\big(w_n(f(x_i)+\varepsilon- m)\big)\mathbb{P}(X_n\in B_{\delta_i}(x_i))\le 1\Big\}\\
	&\le \inf\Big\{m\in\R\colon N \max_{1\le i\le N}\big\{ \exp\big(w_n(f(x_i)+\varepsilon- m)\big)\mathbb{P}(X_n\in B_{\delta_i}(x_i))\big\}\le 1\Big\}\\
	&\le \inf\Big\{m\in\R\colon N  \exp\big(w_n(f(x_{i_n})+\varepsilon- m)\big)\mathbb{P}(X_n\in B_{\delta_{i_n}}(x_{i_n}))\le 1\Big\}\\
	&=\Big\{f(x_{i_n})+\varepsilon - w^{-1}_n\big(- \log N - \log \mathbb{P}(X_n\in B_{\delta_{i_n}}(x_{i_n}))\big)\Big\},
	\end{align*}
	where $i_n:={\rm{arg\,max}}_{1\le i\le N}\big\{\exp\big(w_n(f(x_i)+\varepsilon- m)\big)\mathbb{P}(X_n\in B_{\delta_i}(x_i))\big\}$. 
	By assumption \eqref{eq:asymptotic} and Proposition \ref{prop:wrate} we have
	\begin{align*}
	\overline{\phi}(f)&\le\limsup_{n\to\infty}\max_{1\le i\le N}\Big\{f(x_{i})+\varepsilon -  w^{-1}_n\big(- \log N - \log \mathbb{P}(X_n\in B_{\delta_{i}}(x_{i}))\big)\Big\}\\
	&\le\max_{1\le i\le N}\Big\{f(x_{i})+\varepsilon - \liminf_{n\to\infty} w^{-1}_n\big(- \log N - \log \mathbb{P}(X_n\in B_{\delta_{i}}(x_{i}))\big)\Big\}\\
	&=f(x_{i_0})+\varepsilon + \overline{J}_{B_{\delta_{i_0}}(x_{i_0})}\\
	&\le f(x_{i_0})+\varepsilon - (\overline{I}(x_{i_0})-\varepsilon)\wedge \varepsilon^{-1},
	\end{align*}
	where 
	$i_0:={\rm{arg\,max}}_{1\le i\le N}\big\{ f(x_{i})+\varepsilon + \overline{J}_{B_{\delta_{i}}(x_{i})}\big\}$.
	Since $\varepsilon>0$ was arbitrary, we get $\overline{\phi}(f)\le \sup_{x\in S}\{f(x)-\overline{I}(x)\}$. 
    By definition of $\overline{I}$, it holds $\overline{\phi}(f)\ge \sup_{x\in S}\{f(x)-\overline{I}(x)\}$, which shows
    \eqref{eq:repShortFall}. As a consequence,
    $\overline{\phi}$ is max-stable on $C(S)$.

	In a second step, we assume that $K\neq S$. Fix $x_0\in S\setminus K$. Further, let $(Z_n)_{n\in\mathbb{N}}$ be the sequence of random variables with values in $S$ defined as
	$Z_n(\omega)=X_n(\omega)$ if $\omega\in X^{-1}_n(K)$, and $Z_n(\omega)=x_0$ otherwise.
	Define 
	\[{\phi}_K\colon C_b(K\cup\{x_0\})\to \R,\quad {\phi}_K(f):=\limsup_{n\to\infty}\overline{\phi}_{l_n}\big(f(Z_n)\big).\]
	Since $K\cup \{x_0\}$ is compact, it follows from the first part of the proof that $\phi_K$ is max-stable. Moreover, since for every  $f\in C(K)$ and all $r\in\mathbb{R}$ it holds
	\[
	\overline{\phi}(f 1_K + r 1_{K^c})={\phi}_K(f 1_K + r 1_{\{x_0\}}),
	\]
	it follows that $\overline{\phi}\mid_{C^K(S)}$ is max-stable.
\end{proof}

\begin{example}\label{ex:entropy}
The upper/lower asymptotic entropic risk measure corresponds to the sequence $l_n(x)=\exp(n x)$, $n\in\mathbb{N}$. 
Then, $w_n(x)=n x$ and $w^{-1}_n(y)=\frac{1}{n} y$. 
Inspection shows that $\overline{\phi}$ and $\underline{\phi}$ coincide with the upper and lower asymptotic entropies, respectively, and the condition \eqref{eq:asymptotic} holds. 
From Proposition~\ref{prop:wrate} we obtain
\begin{equation*}\label{eq:Jexp}
\overline{J}_A
=\limsup_{n\to\infty}\frac{1}{n}\log\mathbb{P}(X_n\in A)\quad\mbox{and}\quad 
\underline{J}_A=\liminf_{n\to\infty}\frac{1}{n}\log\mathbb{P}(X_n\in A)
\end{equation*}
for all $A\in\mathcal{B}(S)$. 
Moreover, the minimal rate functions are given by
\begin{equation*}\label{eq:rateexp}
\overline{I}(x)=-\lim_{\delta\downarrow 0}\limsup_{n\to\infty}\frac{1}{n}\log\mathbb{P}(X_n\in B_\delta(x))
\end{equation*}
\begin{equation*}\label{eq:rateexpII}
\underline{I}(x)=-\lim_{\delta\downarrow 0}\liminf_{n\to\infty}\frac{1}{n}\log\mathbb{P}(X_n\in B_\delta(x))
\end{equation*}
for all $x\in S$. Here, the condition (A') corresponds to the exponential tightness of the sequence $(X_n)_{n\in\N}$. 
\end{example}


\begin{example}\label{ex:shortfallBij}
Let $l_n(x):=\exp(n w(x))$ where $w\colon (-\infty,+\infty]\to (-\infty,+\infty]$ is an increasing bijection with $w(0)=0$.  
Then, $w_n(x)=n w(x)$, $w_n^{-1}(y)=w^{-1}(y/n)$, and it can be checked that $w^{-1}_n$ satisfies \eqref{eq:asymptotic}.  
By Proposition \ref{prop:wrate}, we obtain the minimal rate functions
\begin{equation}
\overline{I}(x)=\lim_{\delta\downarrow 0}\liminf_{n\to\infty}w^{-1}\Big(-\frac{1}{n}\log\mathbb{P}(X_n\in B_\delta(x))\Big)
\end{equation}
\begin{equation}
\underline{I}(x)=\lim_{\delta\downarrow 0}\limsup_{n\to\infty}w^{-1}\Big(-\frac{1}{n}\log\mathbb{P}(X_n\in B_\delta(x))\Big)
\end{equation}
for all $x\in S$. 
For instance, 
let $p>0$ be a positive real number, and $w(x):=\rm{sgn}(x)|x|^{p}$. Then,  $w^{-1}(y)=\rm{sgn}(y)|y|^{1/p}$ and the minimal rate functions are given by
\begin{equation}
\overline{I}(x)=\lim_{\delta\downarrow 0}\liminf_{n\to\infty}\Big(-\frac{1}{n}\log\mathbb{P}(X_n\in B_\delta(x))\Big)^{1/p}
\end{equation}
\begin{equation}
\underline{I}(x)=\lim_{\delta\downarrow 0}\limsup_{n\to\infty}\Big(-\frac{1}{n}\log\mathbb{P}(X_n\in B_\delta(x))\Big)^{1/p}
\end{equation}
for all $x\in S$.
\end{example}

\begin{remark}
The previous example can be interpreted as a transformation of the classical LDP, which corresponds to the asymptotic entropy. To that end, suppose that $(X_n)_{n\in\mathbb{N}}$ is exponentially tight, and $\overline{\phi}_{\rm ent}(f)=\underline{\phi}_{\rm ent}(f)$ for all $f\in C_b(S)$. Define the minimal rate function $I_{\min}(x):=\sup_{f\in C_b(S)}\{f(x)-\overline{\phi}_{\rm ent}(f)\}$.   
By Theorem \ref{th:localMaxSt}, the asymptotic entropies $\overline{\phi}_{\rm ent}$ and $\underline{\phi}_{\rm ent}$ satisfy the LDP with rate function $I_{\min}$, which in line with Example \ref{ex:entropy} implies that 
the sequence $(X_n)_{n\in\mathbb{N}}$ satisfies the classical LDP
 \begin{equation}\label{eq:LDPclassical}
-\underset{x\in \textnormal{int}(A)}\inf I_{\min}(x)\le \underset{n\to\infty}\liminf \frac{1}{n} \log\mathbb{P}(X_n\in A)
\le \underset{n\to\infty}\limsup \frac{1}{n} \log\mathbb{P}(X_n\in A)
\le -\underset{x\in\textnormal{cl}(A)}\inf I_{\min}(x)
\end{equation}
for every Borel set $A$; see also Bryc's lemma \citep[Theorem 4.4.2]{zeitouni1998large}.

Let $v\colon(-\infty,+\infty]\to(-\infty,+\infty]$ be an increasing bijection with $v(0)=0$. 
Since $v$ is continuous, by applying $v$ to \eqref{eq:LDPclassical}, we obtain
\begin{align}
-\underset{x\in\textnormal{int}(A)}\inf (v\circ I_{\min})(x)&\le -\underset{n\to\infty}\limsup\, v\Big(-\frac{1}{n} \log\mathbb{P}(X_n\in A)\Big)\nonumber\\ 
&\le -\underset{n\to\infty}\liminf\, v\Big(-\frac{1}{n} \log\mathbb{P}(X_n\in A)\Big)
\le -\underset{x\in \textnormal{cl}(A)}\inf (v\circ I_{\min})(x). \label{eq:perturbedLDP}
\end{align} 
Consider the sequence of functions $w_n(x):=n v^{-1}(x)$. The sequence of the respective inverses $w_n^{-1}(y):=v(y/n)$ satisfies the condition \eqref{eq:asymptotic}. For the corresponding shortfall risk measures $\underline{\phi}$ and  $\overline{\phi}$, it follows from Proposition \ref{prop:perturbedBryc} that the associated minimal rate functions are given by \[\overline{I}=\underline{I}=v\circ I_{\min}.\] 
Also, it can be checked that the exponential tightness of $(X_n)_{n\in\mathbb{N}}$ implies that $\overline{\phi}$ satisfies the condition (A'). Therefore, Proposition~\ref{prop:perturbedBryc} implies
\begin{equation*}\label{eq:perturbedLP}
		\lim_{n\to\infty}\inf\Big\{m\in\mathbb{R}\colon\mathbb{E}\big[\exp\big(n v^{-1}(f(X_n)- m)\big)\big]\le 1 \Big\}=\sup_{x\in S}\big\{f(x)-(v\circ I_{\min})(x)\big\}
\end{equation*}
for all $f\in C_b(S)$.

In particular, the presented theory allows for an explicit form of the LP, which corresponds to the transformed LDP \eqref{eq:perturbedLDP}, by means of asymptotic shortfall risk measures. A natural question is to determine the rate functions associated to asymptotic shortfall risk measure, or even more general max-stable monetary risk measures in case that $(X_n)_{n\in\mathbb{N}}$ is given by the sample means of an i.i.d.~sequence of random variables. This will be part of a forthcoming work. 
\end{remark}

%
%

 \section*{Acknowledgements}
The authors would like to thank Daniel Bartl, Stephan Eckstein, Dieter Kadelka and Ariel Neufeld for helpful discussions and comments. 
They also are grateful to an anonymous referee for a careful review of the manuscript and valuable comments which have improved the presentation of the results. 
The second author was partially supported by the grant Fundaci\'{o}n S\'{e}neca 20903/PD/18.




\begin{thebibliography}{99}

\bibitem[\protect\citeauthoryear{Artzner, P., Delbaen, F., Eber, J.~M. and Heath, D.}{1999}]{artzner01}
Artzner, P., Delbaen, F., Eber, J.~M. and Heath, D. (1999).  
 Coherent measures of risk. \textit{Math. Finance},
\textbf{9} 203--228.
\MR{1850791}

\bibitem[\protect\citeauthoryear{Backhoff-Veraguas, J., Lacker, D. and Tangpi, L.}{2020}]{backhoff2018non}
Backhoff-Veraguas, J., Lacker, D. and Tangpi, L. (2020).  
Non-exponential Sanov and Schilder theorems on Wiener space: BSDEs,
  Schr\"{o}dinger problems and control. 
\textit{Ann. Appl. Probab.},
\textbf{30} 1321--1367.

\bibitem[\protect\citeauthoryear{Bartl, D., Cheridito, P. and Kupper M.}{2019}]{bartl2019robust}
Bartl, D., Cheridito, P. and Kupper M. (2019).  
Robust expected utility maximization with medial limits.  
\textit{J. Math.
Anal. Appl.},
\textbf{471} 752--775.
\MR{3906351}

\bibitem[\protect\citeauthoryear{Bryc, W.}{1990}]{bryc1990large}
Bryc, W. (1990).  
Large deviations by the asymptotic value method.
\textit{Diffusion processes and related problems in analysis},
447--472.
\MR{1110177}

\bibitem[\protect\citeauthoryear{Cattaneo, M.}{2016}]{cattaneo}
Cattaneo, M. (2016).  
On maxitive integration. 
\textit{Fuzzy sets and systems},
\textbf{304} 65--81.
\MR{3553739}

\bibitem[\protect\citeauthoryear{Cheridito, P., Kupper, M. and Tangpi, L.}{2015}]{cheridito2015representation}
Cheridito, P., Kupper, M. and Tangpi, L. (2015).  
Representation of increasing convex functionals with countably
  additive measures. 
\textit{ArXiv} preprint:1502.05763.

\bibitem[\protect\citeauthoryear{Cram\'{e}r, H.}{1938}]{cramer1938nouveau}
Cram\'{e}r, H. (1938).  
Sur un nouveau th{\'e}oreme-limite de la th\'{e}orie des  probabilit\'{e}s. 
\textit{Actual. Sci. Ind.},
\textbf{736} 5--23. 

\bibitem[\protect\citeauthoryear{Dembo, A. and Zeitouni, O.}{1998}]{zeitouni1998large}
Dembo, A. and Zeitouni, O. (1998).  
\textit{Large Deviations Techniques and Applications},
 Springer. 
\MR{1619036}

\bibitem[\protect\citeauthoryear{Dinwoodie, I.~H.}{1993}]{dinwoodie}
Dinwoodie, I.~H. (1993).  
Identifying a large deviation rate function.  
\textit{Ann. Probab.},
\textbf{21} 216--231. 
\MR{1207224}

\bibitem[\protect\citeauthoryear{Donsker, M.~D. and Varadhan, S.~R.~S.}{1975}]{donsker1}
Donsker, M.~D. and Varadhan, S.~R.~S. (1975).  
Asymptotic evaluation of certain Markov process expectations for large time I. 
\textit{Comm. Pure Appl. Math.},
\textbf{28} 1--47. 
\MR{0386024}

\bibitem[\protect\citeauthoryear{Donsker, M.~D. and Varadhan, S.~R.~S.}{1975}]{donsker2}
Donsker, M.~D. and Varadhan, S.~R.~S. (1975).  
Asymptotic evaluation of certain Markov process expectations for large time II. 
\textit{Comm. Pure Appl. Math.},
\textbf{28} 279--301. 
\MR{0386024}

\bibitem[\protect\citeauthoryear{Donsker, M.~D. and Varadhan, S.~R.~S.}{1976}]{donsker3}
Donsker, M.~D. and Varadhan, S.~R.~S. (1976).  
Asymptotic evaluation of certain Markov process expectations for large time III. 
\textit{Comm. Pure Appl. Math.},
\textbf{29} 389--461. 
\MR{0428471}

\bibitem[\protect\citeauthoryear{Donsker, M.~D. and Varadhan, S.~R.~S.}{1976}]{donsker4}
Donsker, M.~D. and Varadhan, S.~R.~S. (1976).  
Asymptotic evaluation of certain Markov process expectations for large time IV. 
\textit{Comm. Pure Appl. Math.},
\textbf{36} 183--212. 
\MR{0690656}

\bibitem[\protect\citeauthoryear{Eckstein, S.}{2019}]{eckstein2019extended}
Eckstein, S. (2019).  
Extended Laplace principle for empirical measures of a Markov chain. 
\textit{Adv. in Appl. Probab.},
\textbf{51} 136--167. 
\MR{3984013}

\bibitem[\protect\citeauthoryear{Freidlin, M.~I. and  Wentzell, A.~D.}{1984}]{freidlin}
Freidlin, M.~I. and  Wentzell, A.~D. (1984).  
\textit{Random Perturbations of Dynamical Systems},
Springer. (Original work published in 1979) (in Russian)

\bibitem[\protect\citeauthoryear{F{\"o}llmer, H. and Knispel, T.}{2011}]{follmer2011entropic}
F{\"o}llmer, H. and Knispel, T. (2011).  
Entropic risk measures: Coherence vs. convexity, model ambiguity and robust large deviations.
\textit{Stoch. Dyn.},
\textbf{11} 333--351. 
\MR{2836530}

\bibitem[\protect\citeauthoryear{F{\"o}llmer, H. and Schied, A.}{2011}]{foellmer01}
F{\"o}llmer, H. and Schied, A. (2011).  
\textit{Stochastic Finance. An Introduction in Discrete Time},
 third revised and extended ed. Berlin: de Gruyter. 
\MR{2779313}

\bibitem[\protect\citeauthoryear{Den Hollander, F.}{2008}]{den2008large}
den Hollander, F. (2008).  
\textit{Large deviations}, 
Amer.~Math.~Soc. 

\bibitem[\protect\citeauthoryear{Kelley, J.~L.}{1975}]{kelley}
Kelley, J.~L. (1975).  
\textit{General topology}, 
Springer. 

\bibitem[\protect\citeauthoryear{Khintchine, A.}{1929}]{khintchine1929neuen}
Khintchine, A. (1929).
{\"U}ber einen neuen Grenzwertsatz der Wahrscheinlichkeitsrechnung.  
\textit{Math. Ann.}, 
\textbf{101} 745--752. 
\MR{1512565}

\bibitem[\protect\citeauthoryear{Lacker, D.}{2020}]{lacker2016non}
Lacker, D. (2020).
A non-exponential extension of Sanov's theorem via convex duality.  
\textit{Adv. in Appl. Probab.}, 
\textbf{52} 61--101. 
\MR{4092808}


\bibitem[\protect\citeauthoryear{Kolokoltsov, V.~N.  and Maslov, V.~P.}{1997}]{maslov}
Kolokoltsov, V.~N.  and Maslov, V.~P. (1997).  
\textit{Idempotent Analysis and its Applications}, 
Springer. 
%

\bibitem[\protect\citeauthoryear{Puhalskii, A.}{2001}]{puhalskii2001large}
Puhalskii, A. (2001).  
\textit{Large deviations and idempotent probability}, 
CRC Press. 


\bibitem[\protect\citeauthoryear{Sanov, I.~N.}{1958}]{sanov1958probability}
Sanov, I.~N. (1958).
On the probability of large deviations of random variables.   
Technical report, North Carolina State University. Dept. of
 Statistics.

\bibitem[\protect\citeauthoryear{Shilkret, N.}{1971}]{shilkret}
Shilkret, N. (1971).
Maxitive measure and integration.    
\textit{Indag. Math.}, 
\textbf{33} 109--116. 
\MR{288225}

\bibitem[\protect\citeauthoryear{Smirnoff, N.}{1933}]{smirnoff1933uber}
Smirnoff, N. (1933).
{\"U}ber Wahrscheinlichkeiten grosser Abweichungen. 
\textit{Matematicheskii Sbornik}, 
\textbf{40} 443--454.
%
\bibitem[\protect\citeauthoryear{Varadhan, S.~R.~S.}{1966}]{varadhan1966asymptotic}
Varadhan, S.~R.~S. (1966).
Asymptotic probabilities and differential equations. 
\textit{Comm. Pure Appl. Math.}, 
\textbf{19} 261--286.  
\MR{203230}

\end{thebibliography}


\end{document}